\documentclass{siamltex}
\usepackage{latexsym}
\usepackage{graphicx,overpic}
\usepackage{amsmath,amsfonts,amssymb,mathtools}
\usepackage{caption}
\usepackage{subcaption}
\usepackage{hyperref}
\usepackage{algorithm2e}
\usepackage[table]{xcolor}
\usepackage[toc]{appendix}

\title{A comparative study of finite element methods for a class of  harmonic map heat flow problems}
\author{Nam Anh Nguyen\thanks{Institut f\"ur
		Geometrie und Praktische  Mathematik, RWTH-Aachen University, D-52056 Aachen,
		Germany; email: {\tt nguyen@igpm.rwth-aachen.de}}
	  \and Arnold Reusken\thanks{Institut f\"ur
		Geometrie und Praktische  Mathematik, RWTH-Aachen University, D-52056 Aachen,
		Germany; email: {\tt reusken@igpm.rwth-aachen.de}}  }

\newcommand*\Laplace{\mathop{}\!\mathbin\bigtriangleup}
\newcommand{\bu}{\mathbf{u}}
\newcommand{\bv}{\mathbf{v}}
\newcommand{\by}{\mathbf{y}}
\newcommand{\bz}{\mathbf{z}}
\newcommand{\bV}{\mathbf{V}}

\newtheorem{remark}{Remark}

\begin{document}
\maketitle

\begin{abstract}
 In this paper, we review and systematically compare three finite element discretization methods for a harmonic map heat flow problem from the unit disk in $\mathbb{R}^2$ to the unit sphere in $\mathbb{R}^3$ in an unified framework. Numerical tests validate the convergence rates  in a regime of smooth solutions and are used to compare the methods in terms of computational efficiency. For one of the methods a discrete inf-sup stability result is derived. 
\end{abstract}

\section{Introduction}
Let $\Omega \subset \Bbb{R}^2$ be a Lipschitz domain and $S^2$ the unit sphere in $\Bbb{R}^3$. In this paper, we study numerical methods for the discretization of the following harmonic map heat flow (HMHF) problem. Given an initial condition $\bu_0: \Omega \to S^2$, determine $\bu(t,\cdot): \Omega \to S^2$ such that
\[
 \partial_t\mathbf{u}= \Laplace \mathbf{u} + \rvert \nabla \mathbf{u} \lvert^2 \mathbf{u}, \quad \bu(0,\cdot)= \bu_0, \quad \bu(t,\cdot)_{|\partial \Omega}=(\bu_0)_{|\partial \Omega}, \quad t \in (0,T]. 
\]
This problem is obtained as the $L^2$ gradient flow of the Dirichlet energy $\int_\Omega |\nabla \bv|^2 \, dx$ on vector fields $\bv: \Omega \to \Bbb{R}^3$ that satisfy a pointwise unit length constraint. Unit lenght minimizers of this Dirichlet energy are called harmonic maps.  The HMHF problem is also closely related to the Landau-Lifshitz-Gilbert (LLG) equation. The HMHF equation can be considered as the limit of the LLG equation where the precessional term vanishes and only damping is left \cite{Melcher2011}. 
Harmonic maps, HMHF and LLG equations have numerous practical applications, for example, in the modeling of ferromagnetic materials or of liquid crystals, cf. e.g.~\cite{Lakshmanan2011,LiuLi2013Skyr,prohl2001,Heinze2013}.

There is an extensive mathematical literature in which topics related to well-posedness, weak formulations, regularity, blow up phenomena and convergence of solutions of the HMHF problem to harmonic maps are studied, cf. e.g~\cite{Struwe85,Struwe08,Eells1964HarmonicMO,Hamilton1975HarmonicMO,KCChang89,Freire95,ChangDingYe1992,vandHout2001}.
 
In this paper, we focus on finite element discretization methods for the HMHF problem. Early work on the development and analysis of numerical methods for HMHF or LLG problems is found in \cite{BarPro06,BaLuPr09,prohl2001,BartelsProhl2007,Alouges2008FEMLL, AlougesJaisson06,Cimrak05,BanasProhlSchaetzle2010}. In recent years, there has been a renewed interest in the numerical analysis of methods for this problem class \cite{Bai2025,KovacsLuebich,BaKoWa22,BBPS24,AkrivisBartelsPalus2025,ABRW25,Santacreu2017,ChenWangXie2021,CaiChenWangXie2022,NguyenReusken,NguyenReuskenFE25}.

 The main difficulties in solving HMHF (or LLG)  numerically are the non-linear term and the fact that the solution has to satisfy the unit length constraint. In the literature, three different finite element based approaches for treating the unit length constraint and the nonlinearity have been proposed. We briefly outline the basic ideas of these approaches and refer to Section~\ref{secFEM} for more details.
 
 In the first approach \cite[Chapter 4]{prohl2001}, \cite{BarPro06,BaLuPr09,RongSun2021,Gui22}, the unit length constraint is enforced weakly by an additional step of renormalizing the numerical solution at every mesh point.  Prohl  \cite[Chapter 4]{prohl2001} shows that the renormalization step can be interpreted as a penalizing term in the equation.  This renormalization leads to a nonlinear system  that has to be solved in each time step, cf.~ \cite{BaLuPr09}. In \cite[Section 4.3.2]{prohl2001} and \cite{RongSun2021,Gui22}, the nonlinearity is treated by a simple extrapolation in time in a semi-implicit Euler scheme which results in a linear system in each time step.  In the recent paper \cite{Bai2025}  this method is extended to more general harmonic map heat flows from a domain $\Omega \subset \Bbb{R}^d$ into a closed smooth  hypersurface in $\Bbb{R}^n$ and using a higher order BDF time stepping scheme.
 
In the second approach \cite{Alouges2008FEMLL, AlougesJaisson06,KovacsLuebich,AkrivisBartelsPalus2025,ABRW25}, the   unit length constraint is reformulated as an orthogonality condition between the solution and its time derivative. This leads to  a solution space tangential to the sphere and the finite element space used in the discretization  mimics this. In this approach one uses semi-implict BDF time stepping schemes in which a straightforward extrapolation is used to treat the nonlinearity. 

In the third approach \cite{BartelsProhl2007}, the double cross product reformulation of the HMHF (or LLG) equation is used, cf.~\eqref{eq:HMHF_PDE_DoubleCross} below. This reformulation has the nice property that the unit length constraint is implicit in the formulation, meaning that its solution  preserves the unit length of the initial condition. In the finite element spatial discretization, the norm preserving structure is mirrored. For treating the nonlinearity that arises in this double cross product reformulation the authors of  \cite{BartelsProhl2007} propose a fixed point iteration.  This approach has also been applied to the $p$-harmonic map heat flow \cite{BarPro08}. Below, cf. Section~\ref{sec:HMHF_Bartels_Method}, besides this fixed point linearization we also study a Newton linearization.

Although for all three approaches one finds papers in which results of numerical experiments with methods based on these are presented, we are not aware of any literature in which a systematic comparison of these three very different techniques is presented. 
 
The main contribution of this paper is twofold. Firstly, we perform a systematic numerical study and compare convergence and efficiency properties of the above-mentioned methods. Secondly, the paper has review character in the sense that we present these methods in an unified framework.  
For the comparative study we focus on the smooth regime, i.e. HMHF problems with a globally smooth solution.  In the smooth regime, for a given (standard) finite element space and consistency order of the  time stepping scheme it is known what the optimal rates of convergence (w.r.t. mesh size and time step size) of the different methods are, which then allows a fair comparison of these methods. In the blow up case, singularities develop and a fair comparison of computational efficiencies becomes much more difficult. For the HMHF, exact solutions are not known, except for very special cases. Hence, for the method comparison we have to determine reliable and highly accurate reference solutions that can be used as sufficiently accurate approximations of the exact solution.   We  therefore restrict to a class of initial data in  \eqref{HMHFeq}, for which global in time smoothness is guaranteed and, due to rotational symmetry, the problem can be transformed to a spatially \emph{one}-dimensional problem. Using the latter property we can determine very accurate reference solutions by numerically solving this lower dimensional problem. \\
 As an outlook to further research we also present results of  a numerical experiment in which the three methods that we compare are applied to a HMHF problem with finite time blow up.
 As a further contribution of this paper we present a discrete inf-sup stability result for the second approach, cf. Lemma~\ref{leminfsup}  below. This result is relevant for an efficient implementation of the method.
 
 The rest of the paper is organized as follows. In Section \ref{sec:HMHF}, we  describe the harmonic map heat flow problem and its formulation for a radially symmetric case. Section \ref{secFEM} describes the finite element discretizations for this radially symmetric case and for the general harmonic map heat flow problem in the respective subsections. In each subsection, we also present  numerical results to validate the convergence rates. The source code used in  these experiments can be found in \cite{nguyen_2025_15481333} and is based on the software package \emph{Netgen/NGSolve}. Results of a numerical experiment for a problem with finite time blow up are presented in Section~\ref{SecOutlook}. In Section \ref{SecComparison} we draw  conclusions and discuss computational efficiency aspects.
 
\section{The harmonic map heat flow problem} \label{sec:HMHF}
Let 
\begin{align*}
         D &= \{x = (x_1,x_2) \in \mathbb{R}^2 : |x| < 1\} \\
        S^2 &= \{x = (x_1,x_2,x_3) \in \mathbb{R}^3 : |x| = 1\} 
\end{align*}
be the open unit disk in $\mathbb{R}^2$ and the unit sphere in $\mathbb{R}^3$, respectively. Then, given $\mathbf{u}_0 \in C^2(\Bar{D},S^2)$ and some $T>0$, we consider the following initial-boundary value problem
\begin{equation} \label{HMHFeq} 
\begin{split}
    \partial_t\mathbf{u}&= \Laplace \mathbf{u} + \rvert \nabla \mathbf{u} \lvert^2 \mathbf{u} \quad \text{ on } D, t \in (0,T],  \\
    \mathbf{u}(0,x) &= \mathbf{u}_0(x) \quad \text{ for } x \in D, 
\end{split}
\end{equation}
which is the harmonic map heat flow problem. We prescribe the Dirichlet boundary condition
\begin{align}
    \mathbf{u}(t,x) &= \mathbf{u}_0(x) \text{ for } x \in \partial D, t \in [0,T]. \label{eq:HMHF_DirichletBoundary}
\end{align}
The PDE \eqref{HMHFeq} is the  $L^2$-gradient flow of the energy
\begin{align}
    E(\mathbf{u}):=\frac{1}{2} \int_D \left|\nabla \mathbf{u}\right|^2 dx \label{eq:HMHF_Energy}
\end{align} 
for vector fields satisfying a pointwise unit length condition. 
Stationary solutions of \eqref{HMHFeq} are called harmonic maps and are critical points of \eqref{eq:HMHF_Energy}.
\\
In this paper  we study three different finite element discretization methods for   \eqref{HMHFeq} in a regime where we have a global in time smooth solution. For computing discretization errors we need a sufficiently accurate approximation of the exact solution of  \eqref{HMHFeq}. This motivates why we restrict to a class of initial data in  \eqref{HMHFeq}, for which global in time smoothness is guaranteed and, due to rotational symmetry, the problem can be transformed to a spatially \emph{one}-dimensional problem. Due to the latter property we can determine very accurate reference solutions by numerically solving this lower dimensional problem, cf. Section~\ref{sec:FEM_RSHMHF} below.
This class of smooth rotationally symmetric solutions is obtained  from the following fundamental result. 
\begin{theorem}[Chang-Ding \cite{Chang1991}] \label{theo:HMHF_1D_Global_Smooth}
Using polar coordinates $(r,\psi)$ on the disk $D$, let the initial condition have the form
\begin{equation}
   \mathbf{u}_0(r,\psi) =
   \begin{pmatrix}  \cos{\psi} \sin{ u_0(r) } \\  \sin{\psi} \sin{ u_0(r) } \\ \cos{u_0(r)}\end{pmatrix}   \label{eq:HMHF_InitialCondSphericallySymmetric}
\end{equation}
where $u_0 \in C^2([0,1])$ with $\|u_0 \|_{L^\infty} \leq \pi$ is given, then there exists a unique global smooth solution to \eqref{HMHFeq} of the form
\begin{equation}
    \mathbf{u}(t,r,\psi) =
   \begin{pmatrix}  \cos{\psi} \sin{ u(t,r) } \\  \sin{\psi} \sin{ u(t,r) } \\ \cos{u(t,r)}\end{pmatrix}. \label{eq:HMHF_SphericallySymmtricSolution} 
\end{equation}
\end{theorem}
The spherically symmetric or corotational form of the solution \eqref{eq:HMHF_SphericallySymmtricSolution} reduces the harmonic map heat flow \eqref{HMHFeq} to the one-dimensional problem
\begin{equation}\label{eq:HMHF_1DPDE} \begin{split}
    \partial_t u &= \partial_{rr}u + \frac{1}{r}\partial_r u - \frac{\sin{(2 u)}}{2r^2} \quad \text{ for } r \in (0,1), t \in (0,T],  \\
    u(0,r) &= u_0(r) \quad \text{ for } r \in (0,1), 
\end{split}
\end{equation}
with $u_0(0)= 0$ where we prescribe the Dirichlet boundary conditions
\begin{align}
    u(t,0) =u_0(0), \; u(t,1)=u_0(1) \text{ for }  t \in [0,T].\label{eq:HMHF_1DDirichletBoundary}
\end{align}
We will call \eqref{eq:HMHF_1DPDE} the corotational harmonic map heat flow (CHMHF).
The energy of solutions of \eqref{eq:HMHF_1DPDE} is then given by \cite[Equation (1.4)]{vdBergHulshofKing2003}
\begin{align}
    E(u) = \pi \int_0^1 r (\partial_r u)^2+ \frac{\sin^2{(u)}}{r} \, dr. \label{eq:HMHF_1DEnergy}
\end{align}

\section{Finite element discretization methods} \label{secFEM}
In this section, we describe discretization methods for the HMHF problem. We apply the method of lines approach in which for spatial discretization we use finite element methods and for discretization in time a low order BDF (BDF1 or BDF2) method is used. In Section~\ref{sec:FEM_RSHMHF}, we consider the CHMHF \eqref{eq:HMHF_1DPDE}.  We propose a basic method for which  optimal order convergence rates in $L^2$- and $H^1$-norms are demonstrated numerically. 
We use this method to construct highly accurate numerical reference solutions  for the HMHF problem \eqref{HMHFeq}. These reference solutions are used to compute discretization errors for the three methods that are treated in Section~\ref{sec:FEM_HMHF}.
\\
In  Section~\ref{sec:FEM_HMHF} we consider three very different finite element discretizations for HMHF \eqref{HMHFeq} based on the schemes presented in  \cite[Chapter 4]{prohl2001},   \cite{KovacsLuebich} and  \cite{BartelsProhl2007}. These methods  differ in how the unit length constraint of HMHF is treated.  We will study convergence rates and efficiency aspect of these methods. The source code for the implementation of all the finite element methods can be found in \cite{nguyen_2025_15481333}.

\subsection{Method for the radially symmetric case} \label{sec:FEM_RSHMHF}
We  propose a very elementary method of lines discretization  for the CHMHF \eqref{eq:HMHF_1DPDE} using  linear or quadratic finite elements in space and BDF1 or BDF2 for time discretization. Optimal order convergence of this method is demonstrated below.
In \cite{HaynesHuangZegeling2013}, a finite difference based scheme with a moving mesh ansatz is developed for CHMHF in the singular regime with blow up of solutions in finite time. We are not aware of other literature in which discretization schemes for CHMHF are systematically studied. We note that due to the strong nonlinearity of CHMHF an error analysis of discretization methods for this problem is not straightforward, cf. the recent work \cite{NguyenReusken,NguyenReuskenFE25}.

Let an initial condition $u_0 \in C^2(\bar I)$, $I:=(0,1)$, with $u_0(0)=0$ and $\|u_0\|_{L^\infty} \leq \pi$ be given. We introduce the solution
space $H^1_{u_0}(I):=\{\, v \in H^1(I)~|~v(0)=0,~v(1)=u_0(1)\, \}$. 

A variational formulation of \eqref{eq:HMHF_1DPDE}-\eqref{eq:HMHF_1DDirichletBoundary} is as follows: find $u \in C^1\big([0,T];H^1_{u_0}(I)\big)$ with $u(0,\cdot)=u_0$ such that for all $v \in H^1_0(I)$
\begin{align}
    \left(\partial_t u, v\right)+ \left(\partial_r u, \partial_r v\right)-\left(\frac{1}{r}\partial_r u, v\right) + \left(\frac{\sin(2u)}{2 r^2},v\right) = 0, \quad t \in (0,T], \label{eq:HMHF_1D_WeakForm}
\end{align}
where $(\cdot,\cdot)$ denotes the $L^2$-scalar product on $I$. Below in the time discretization, we use a linearization of the nonlinear term in \eqref{eq:HMHF_1D_WeakForm} based on
\[
 \frac{\sin(2u)}{2 r^2}=  \left(\frac{\sin{(2u)}}{2 u\,  r^2}\right) u.
\]
For simplicity, we use  an equidistant partitioning  of $\Bar{I}$ into $N$ subintervals of equal length  $h = 1/N$, denoted by $\mathcal{T}_h$. We introduce the finite element spaces of continuous and piecewise polynomial functions of degree less than or equal to $p$:
\begin{align*}
    V_h^p &= \{v \in C^0(\Bar{I}):  v|_K \in \mathcal{P}_p \; \forall K \in \mathcal{T}_h \}, \\
     V^p_{h,u_0} &=  \{v \in V_h^p: v(0) =0, v(1)= u_0(1)\}.
\end{align*}
 In the experiments below we restrict to  $p=1$ and $p=2$. 
\\
Let $\tau$ be a constant time step and $J\in \Bbb{N}$ such that $J \tau = T$.
For BDF1 (implicit Euler) and BDF2, we use the standard compact representation
\begin{align}
    \partial_t u(t_j,\cdot) &\approx \frac{1}{\tau} \left( u^{j}-u^{j-1} \right) =: \frac{1}{\tau} \sum_{i=0}^{1} \delta_i^{(1)} u^{j-i}, \notag \\
    \partial_t u(t_j,\cdot) &\approx \frac{1}{\tau} \left( \frac{3}{2}u^{j}-2u^{j-1}+\frac{1}{2} u^{j-2} \right) =: \frac{1}{\tau} \sum_{i=0}^{2} \delta_i^{(2)} u^{j-i}, \label{eq:HMHF_BDF_Def}
\end{align}
where the superscript $k\in\{1,2\}$ in $\delta_i^{(k)}$ corresponds to BDF$k$.  The extrapolation used in the linearization is given by
\begin{align}
	\begin{split}
		\Hat{u}^{j,(1)}& :=u^{j-1}   \qquad \qquad \text{(BDF1)},  \\
		\Hat{u}^{j,(2)}& :=2u^{j-1}-u^{j-2} \; \text{(BDF2)}.
	\end{split}
   	\label{eq:HMHF_Approx_Extrapolation} 
\end{align}
 Let $\mathcal{I}_h$ be the nodal interpolation operator from $C^0(\bar{I},\mathbb{R})$ to $V_h^p$. We fix $k \in \{1,2\}$ and a timestep $\tau$. 
The discretization is as follows: Given $u_{h}^0 =\mathcal{I}_h( u_0)$ and $u_{h}^1\in V_{h,u_0}^p$ (only for $k=2$), for $j=k,\ldots, J-1$,  find $u_h^{j+1}\in V_{h,u_0}^p$ such that for all $v_h \in V^p_{h,0}$
\begin{align}
    \frac{\delta_0^{(k)}}{\tau} \left(u_h^{j+1},v_h\right) &+ \left(\partial_r u_h^{j+1}, \partial_r v_h\right)-\left(\frac{1}{r}\partial_r u_h^{j+1}, v_h\right) \notag \\
    &+ \left(\frac{\sin{(2\Hat{u}_h^{j+1,(k)})}}{2r^2 \Hat{u}_h^{j+1,(k)}} u_h^{j+1} ,v_h\right) =- \frac{1}{\tau} \left(\sum_{i=1}^k\delta_i^{(k)} u_h^{j-i},v_h\right). \label{eq:HMHF_1D_full_discr}
\end{align}
For $k=2$ the initial value $u_{h}^1$ is determined by applying one step of the BDF1 scheme. 
An error analysis of the  scheme \eqref{eq:HMHF_1D_full_discr} for $k=1$ is given in \cite{NguyenReuskenFE25}. In \cite{NguyenReusken},  an error analysis of a finite difference scheme for \eqref{eq:HMHF_1DPDE} is presented.
\paragraph*{Numerical results} 
We consider the CHMHF for $t \in [0,10^{-1}]$ and with initial condition
\begin{align}
    u_0(r)=\frac{1}{2}\pi r^2.\label{eq:HMHF_1DPDE_InitialCond_RefSol}
\end{align}
Note that this initial condition satisfies $\|u_0(r)\|_{L^\infty} < \pi$ for all $r\in [0,1]$ and thus Theorem \ref{theo:HMHF_1D_Global_Smooth} ensures that there is no blow up and the exact solution is globally smooth. As a reference solution, we take the numerical solution $u_h^{\text{ref}}$ computed using \eqref{eq:HMHF_1D_full_discr} with $h = 2^{-14}$, piecewise quadratic finite elements ($p=2$),  $\tau = 10^{-6}$ and BDF2 ($k=2$).  We compute errors in $u_h^J\approx u(0.1, \cdot)$ in the $L^2$ and $H^1$ norms, i.e.,  $\lVert u_h^J - u_h^{\text{ref}}\rVert_{L^2(I)}$ and $\lVert u_h^J - u_h^{\text{ref}}\rVert_{H^1(I)}$.  Results are presented in the 
Tables~\ref{table:HMHF_1D_FEM_L2_H1_Convergence_tau_bdf1}--\ref{table:HMHF_1D_FEM_L2_H1_Convergence_h_p2}. 
\begin{table}[ht!]
\centering
\begin{tabular}{ |p{2.4cm}|| p{2.65cm}| p{1cm}|| p{2.65cm}|p{1cm}|}
 \hline
  $h = 2^{-14}$ & $\lVert u_h^{J} - u^{\text{ref},J}_h \rVert_{L^2(I)}$ & EOC & $\lVert u_h^{J} - u^{\text{ref},J}_h \rVert_{H^1(I)}$ & EOC\\
 \hline
 $\tau=5\cdot 10^{-2}$  & $ 4.2056e-02$  &  $-$ & $1.3876e-01$ & $-$  \\
 $\tau=2.5 \cdot 10^{-2}$  & $2.3299e-02$ & $0.85$  & $7.6855e-02$ & $0.85$ \\
 $\tau=1.25 \cdot 10^{-2}$ & $1.2331e-02$ & $0.92$ & $4.0678e-02$ & $0.92$\\
 $\tau=6.25 \cdot 10^{-3}$ & $6.3542e-03$ & $0.96$ & $2.0964e-02$ & $0.96$\\
$\tau = 3.125 \cdot 10^{-3}$ & $3.2269e-03$ & $0.98$  & $1.0647e-02$& $0.98$\\
 \hline 
\end{tabular}
\caption{Error for \eqref{eq:HMHF_1D_full_discr} with BDF1 and $V_h^2$; $h= 2^{-14}$ fixed. }
\label{table:HMHF_1D_FEM_L2_H1_Convergence_tau_bdf1}
\end{table}

\begin{table}[ht!]
\centering
\begin{tabular}{ |p{2.4cm}|| p{2.65cm}| p{1cm}|| p{2.65cm}|p{1cm}|}
 \hline
  $h = 2^{-14}$& $\lVert u_h^{J} - u^{\text{ref},J}_h \rVert_{L^2(I)}$ & EOC & $\lVert u_h^{J} - u^{\text{ref},J}_h \rVert_{H^1(I)}$ & EOC\\
 \hline
 $\tau=5\cdot 10^{-2}$  & $ 2.5176e-02$  &  $-$ & $8.3363e-02$ & $-$  \\
 $\tau=2.5 \cdot 10^{-2}$  & $5.6626e-03$ & $2.15$  & $1.9086e-02$ & $2.13$ \\
 $\tau=1.25 \cdot 10^{-2}$ & $1.0739e-03$ & $2.40$ & $3.6513e-03$ & $2.39$\\
 $\tau=6.25 \cdot 10^{-3}$ & $2.3314e-04$ & $2.20$ & $7.8868e-04$ & $2.21$\\
$\tau = 3.125 \cdot 10^{-3}$ & $5.5458e-05$ & $2.07$  & $1.8704e-04$& $2.08$\\
 \hline 
\end{tabular}
\caption{Error for \eqref{eq:HMHF_1D_full_discr} with BDF2  and $V_h^2$; $h= 2^{-14}$ fixed.}
\label{table:HMHF_1D_FEM_L2_H1_Convergence_tau_bdf2}
\end{table}

The results in Tables~\ref{table:HMHF_1D_FEM_L2_H1_Convergence_tau_bdf1} and \ref{table:HMHF_1D_FEM_L2_H1_Convergence_tau_bdf2}  show that the scheme \eqref{eq:HMHF_1D_full_discr} has optimal order of convergence  in the time step $\tau$ for BDF1 and BDF2. 
\begin{table}[ht!]
\centering
\begin{tabular}{ |p{2.4cm}|| p{2.65cm}| p{1cm}|| p{2.65cm}|p{1cm}|}
 \hline
  $\tau = 10^{-6}$& $\lVert u_h^{J} - u^{\text{ref},J}_h \rVert_{L^2(I)}$ & EOC & $\lVert u_h^{J} - u^{\text{ref},J}_h \rVert_{H^1(I)}$ & EOC
 \\
 \hline
 $h = 2^{-3}$  & $ 1.0893e-03$  &  $-$ & $2.4522e-02$ & $-$ \\
 $h = 2^{-4}$  & $2.7888e-04$ & $1.97$ & $1.2328e-02$ & $0.99$ \\
 $h = 2^{-5}$ & $6.9802e-05$ & $2.00$ & $6.1812e-03$ & $1.00$ \\
 $h = 2^{-6}$ & $1.7100e-05$ & $2.03$ & $3.0995e-03$& $1.00$\\
$h = 2^{-7}$ & $3.9710e-06$ & $2.11$ & $1.5422e-03$ & $1.01$\\
 \hline 
\end{tabular}
\caption{Error for \eqref{eq:HMHF_1D_full_discr} with $V_h^1$ and BDF2; $\tau = 10^{-6}$ fixed.}
\label{table:HMHF_1D_FEM_L2_H1_Convergence_h_p1}
\end{table}
\begin{table}[ht!]
\centering
\begin{tabular}{ |p{2.4cm}|| p{2.65cm}| p{1cm}|| p{2.65cm}|p{1cm}|}
 \hline
   $\tau = 10^{-6}$& $\lVert u_h^{J} - u^{\text{ref},J}_h \rVert_{L^2(I)}$ & EOC & $\lVert u_h^{J} - u^{\text{ref},J}_h \rVert_{H^1(I)}$ & EOC
 \\
 \hline
 $h = 2^{-3}$  & $ 3.5116e-05$  &  $-$ & $2.0124e-03$ & $-$ \\
 $h = 2^{-4}$  & $4.1838e-06$ & $3.07$ & $5.0027e-04$ & $2.01$ \\
 $h = 2^{-5}$ & $5.1173e-07$ & $3.03$ & $1.2469e-04$ & $2.00$ \\
 $h = 2^{-6}$ & $6.3310e-08$ & $3.01$ & $3.1125e-05$& $2.00$\\
$h = 2^{-7}$ & $7.8741e-09$ & $3.01$ & $7.7756e-06$ & $2.00$\\
 \hline 
\end{tabular}
\caption{Error for \eqref{eq:HMHF_1D_full_discr} with $V_h^2$ and  BDF2; $\tau = 10^{-6}$ fixed.}
\label{table:HMHF_1D_FEM_L2_H1_Convergence_h_p2}
\end{table}

In Tables~\ref{table:HMHF_1D_FEM_L2_H1_Convergence_h_p1} and \ref{table:HMHF_1D_FEM_L2_H1_Convergence_h_p2} we observe optimal order convergence rates  with respect to $h$, in $L^2$- and $H^1$-norms, for  $p=1$ and $p=2$. Below the scheme \eqref{eq:HMHF_1D_full_discr} is used to determine a highly accurate reference solution for the HMHF problem \eqref{HMHFeq}, cf. Remark~\ref{rem:RSHMHF_SphericalTransf}. 

\subsection{Methods for the harmonic map heat flow} \label{sec:FEM_HMHF}
In this section, we study discretization methods for the  HMHF problem \eqref{HMHFeq}. They are based on the works of Prohl \cite[Section 4.3.2]{prohl2001}, Akrivis \emph{et al. }\cite{KovacsLuebich} and Bartels-Prohl \cite{BartelsProhl2007} and differ in how the unit length constraint of HMHF is treated.  In  \cite[Section 4.3.2]{prohl2001}, the unit length constraint is weakly enforced by normalizing the numerical solution at every mesh point after each time step. A different way, proposed in \cite{BartelsProhl2007}, is to use  the double cross product reformulation of HMHF, which  automatically preserves the unit length of the initial condition, without using an additional unit length constraint.  The method studied in \cite{KovacsLuebich} is based on the approach introduced in \cite{AlougesJaisson06,Alouges2008FEMLL} and exploits the fact that preserving the unit length constraint can be formulated as the time derivative of the solution being orthogonal to the solution itself. Motivated by this, one constructs an approximate tangent space to the sphere as the solution space for the time derivative to enforce the unit length.  This idea of solving for the approximate time derivative in the finite element discretization goes back to  \cite{AlougesJaisson06,Alouges2008FEMLL} where a pointwise orthogonality in the mesh points is used to define a discrete tangent space. In  \cite{KovacsLuebich}, an $L^2$-projection of the orthogonality condition into the finite element space is used.

Since the domain $D$ has a curved boundary, for the case of a quadratic finite element space $V_h^2$ we will use isoparametric elements, which is essential for obtaining  optimal order of convergence, cf. \cite{LenoirIsoFE86}. We denote by ${\mathcal{T}}_h$ a regular and quasi-uniform  triangulation of ${D}$ with maximum mesh parameter $h$. Note that ${\mathcal{T}}_h$ is not related to the interval partitioning (also denoted by ${\mathcal{T}}_h$) used in Section~\ref{sec:FEM_RSHMHF} above. The vertices of the boundary triangles are assumed to lie on $\partial D$.  The set of vertices in $\mathcal{T}_h$ is denoted by  $\mathbf{\mathcal{N}}_{{h}}$.
We introduce finite element spaces of continuous  piecewise polynomial functions of maximal degree $p$:
\begin{align}
    V_{h}^p &= \{v \in C(D):  v|_K \in \mathcal{P}_p \; \forall K \in {\mathcal{T}}_h \}, \label{eq:HMHF2D_fes_1} \\
    \mathbf{V}_{h}^p &= [V^p_{{h}}]^3, \label{eq:HMHF2D_fes_2}  \\
     \mathbf{V}^p_{{h},\mathbf{u}_0} &= \{ \mathbf{v} \in \mathbf{{V}}^p_{{h}} : \mathbf{v}(z) = \mathbf{u}_0(z) ~~\forall ~z \in \mathbf{\mathcal{N}}_{{h}} \cap \partial D\, \}. \label{eq:HMHF2D_fes_3}
\end{align}
In the numerical experiments, we restrict to $p=1$ and $p=2$. For $p=2$ the isoparametric variants of these spaces are used.
\begin{remark} \label{rem:RSHMHF_SphericalTransf} \rm 
We define the  mapping $F:C(I) \rightarrow C(D;\mathbb{R}^3)$ by
\begin{align*}
    F(u)(x,y) = \begin{pmatrix}  x/r \sin{ u(r) } \\  y/r \sin{ u(r) } \\ \cos{u(r)}\end{pmatrix}, \quad r = \sqrt{x^2+y^2}, \quad (x,y) \in D.
\end{align*}
This mapping is used to map a reference solution $u_h^{\text{ref}}$, cf. Section~\ref{sec:FEM_RSHMHF}, to a corresponding $\mathbf{u}^{\text{ref}}_{h}$ for HMHF as follows. Given $u_h^{\text{ref}}$ at a fixed time $t$ we compute
\begin{align}
    \mathbf{u}^{\text{ref}}_{h} = \mathcal{I}_{h}\label{eq:RSHMHF_SphericalTransformation} \left(F\left(u_h^{\text{ref}}\right)\right)
\end{align}
where $ \mathcal{I}_{h} $ is the nodal interpolation operator from $C^0(\Bar{D};\mathbb{R}^3)$ to $\mathbf{V}_{h}^p$.
\end{remark}
\subsubsection{Treatment of unit length constraint by pointwise projection} \label{sec:HMHF_PPFEM}
We recall the method from \cite[Section 4.3.2]{prohl2001}. A variational formulation of \eqref{HMHFeq} reads as follows: Given $\mathbf{u}_0 \in H^1({D};\mathbb{R}^3) $ with $\left|\mathbf{u}_0\right|=1$ (a.e),  find $\mathbf{u} \in C^1\big([0,T]; H^1({D};\mathbb{R}^3)\big)$ with $\mathbf{u}(0,\cdot)=\mathbf{u}_0$ and $\mathbf{u}_{|\partial D}=(\mathbf{u}_0)_{|\partial D}$ such that for all $\mathbf{v} \in H^1_0(D;\mathbb{R}^3)$
\begin{align*}
    &\left(\partial_t \mathbf{u}, \mathbf{v}\right) + \left(\nabla \mathbf{u},\nabla \mathbf{v} \right) - \left(\left| \nabla \mathbf{u} \right|^2 \mathbf{u}, \mathbf{v} \right) = 0, \quad t \in (0,T], \\
    &|\mathbf{u}| = 1 \text{ a.e. in } [0,T] \times {D}, 
\end{align*}
where $(\cdot,\cdot)$ denotes the $L^2$-scalar product on $D$. Let the time step  $\tau$ and $J \in  \Bbb{N}$ be such that $J \tau = T$. Based on this variational formulation we obtain the following discrete problem, cf. \cite[Section 4.3.2]{prohl2001}, with given $k \in\{1,2\}$ and time step $\tau$: Given $\mathbf{u}_{h}^{0}=\mathcal{I}_{h}(\mathbf{u}_0)$ and $\mathbf{u}_{h}^{1} \in \bV_{h,\bu_0}^p$ (only for $k=2$), for $j=k-1, \ldots, J-1$,  find ${\mathbf{u}}_{h}^{j+1} \in \mathbf{{V}}^p_{{h},\mathbf{u}_0}$ such that for all $\mathbf{v}_h \in \mathbf{V}^p_{{h},0}$:
\begin{align}
    \frac{\delta_0^{(k)}}{\tau}\left( \Tilde{\mathbf{u}}^{j+1}_{h}, \mathbf{v}_{h} \right) + \left(\nabla \Tilde{\mathbf{u}}^{j+1}_{h}, \nabla \mathbf{v}_{h}\right) - \left(\left| \nabla \Hat{\mathbf{u}}^{j+1,(k)}_{h} \right|^2 \Tilde{\mathbf{u}}^{j+1}_{h}, \mathbf{v}_{h}\right) \notag\\
    = -\frac{1}{\tau}\sum_{i=1}^k\left(\delta_i^{(k)} \mathbf{u}^{j+1-i}_{h}, \mathbf{v}_{h} \right)\; ,\label{eq:HMHF_PPFE}
    \\
     \mathbf{u}_{h}^{j+1}(z) = \frac{\Tilde{\mathbf{u}}^{j+1}_{h}(z)}{\left| \Tilde{\mathbf{u}}^{j+1}_{h}(z)\right|} \quad \text{for all}~ z \in \mathbf{\mathcal{N}}_{h},\label{eq:HMHF_PPFE_normalization}
\end{align}
where $\Hat{\mathbf{u}}^{j+1,(k)}_{h}$ is the extrapolation given by the formula \eqref{eq:HMHF_Approx_Extrapolation} and $\delta_i^{(k)}$ are the coefficients of the time stepping scheme \eqref{eq:HMHF_BDF_Def}. 
For $k=2$, the initial function $\mathbf{u}_{h}^{1}$ is determined by applying one step of the BDF1 ($k=1$) variant of the scheme \eqref{eq:HMHF_PPFE}-\eqref{eq:HMHF_PPFE_normalization}. Note that the norm constraint is enforced by an additional pointwise normalization step. 
We will call the scheme \eqref{eq:HMHF_PPFE}-\eqref{eq:HMHF_PPFE_normalization} the pointwise projection finite element method (PPFEM) for the harmonic map heat flow. 
\begin{remark} \rm 
Under the assumption that the solution is sufficiently smooth and a parameter constraint $\tau^{-1/2}= o(h^{-1})$ is satisfied, it is shown in \cite[Theorem 4.11]{prohl2001}  that a variant of this scheme, modified for the Landau-Lifschitz-Gilbert equation, converges with first order in $\tau$ for BDF1 $(k=1)$. The author also proves  for linear finite elements first order convergence in $h$ in the $H^1$-norm. Recently, in \cite{Gui22} a slight modification of PPFEM is proposed, using a lumped mass finite element method. For this  lumped mass PPFEM it is proved  that on  rectangular meshes and under the assumption $h^{p+1} = \mathcal{O}(\tau)$, one obtains optimal convergence order in $h$ and in $\tau$ in the $L^2$-norm,  using BDF1 for time discretization and linear or higher order conforming finite element spaces. In the case of triangular meshes, one order in space is lost.  In \cite{Bai2025} the method is extended to more general harmonic map heat flows from a domain $\Omega \subset \Bbb{R}^d$ into a closed smooth  hypersurface in $\Bbb{R}^n$. Furthermore, in that paper  a higher order BDF  stepping scheme is used for time discrezation $(k=2,\ldots,5)$. On rectangular domains, using $H^1$-conforming  tensor product finite element spaces, they show optimal convergence order $\mathcal{O}(h^{p+1} + \tau^k)$ in $L^2$ under the condition $h^{p+1} = \mathcal{O}(\tau^k)$. For simplicial finite element spaces on polygonal domains one order in space is lost,  as in \cite{Gui22}. Both papers also provide numerical results for the convergence behavior using a two dimensional rectangular domain (hence, no spherical symmetry). The lumped mass technique is also used in the method treated in \cite{BartelsProhl2007}, cf. Section \ref{sec:HMHF_Bartels_Method} below. In \cite{RongSun2021}, the discretization method described here is analyzed for LLG and they derive, for $p \geq 2$, optimal convergence order $\mathcal{O}(\tau + h^{p+1})$ in the $L^2$-norm under the time step condition  $\tau = \mathcal{O}(h)$.
\end{remark} 
\paragraph*{Numerical results}
We consider the HMHF problem \eqref{HMHFeq} for $t \in [0,10^{-1}]$ and with initial condition
\begin{equation}
   \mathbf{u}_0(r,\psi) =
   \begin{pmatrix}  \cos{\psi} \sin{ u_0(r) } \\  \sin{\psi} \sin{ u_0(r) } \\ \cos{u_0(r)}\end{pmatrix},   \label{eq:HMHF_InitialCondSphericallySymmetric_example}
\end{equation}
where $(r,\psi)$ are the polar coordinates on the disk and $u_0$ as given by \eqref{eq:HMHF_1DPDE_InitialCond_RefSol}. 
The numerical solution $\mathbf{u}_{h}^j, j=0,\dots, J$, is determined by PPFEM \eqref{eq:HMHF_PPFE}-\eqref{eq:HMHF_PPFE_normalization}.
For computing errors we use the reference solution  $u_h^{\text{ref}}$ computed using the method explained in Section~\ref{sec:FEM_RSHMHF} (with $h = 2^{-14}$, $V_h^2$, BDF2 with $\tau = 10^{-6}$). The reference solution $\mathbf{u}_{h}^{\text{ref}}$ is constructed from  $u_h^{\text{ref}}$ by \eqref{eq:RSHMHF_SphericalTransformation}. We compute the errors
in $\mathbf{u}_h^J \approx \mathbf{u}(0.1,\cdot)$ in the $L^2$ and $H^1$ norms, i.e., $\lVert \mathbf{u}_h^J - \mathbf{u}_h^{\text{ref}}\rVert_{L^2}$ and $\lVert \mathbf{u}_h^J - \mathbf{u}_h^{\text{ref}}\rVert_{H^1}$. \\
Results are presented in the Tables~\ref{table:HMHF_2D_PPFEM_L2_H1_Convergence_tau_bdf1}--\ref{table:HMHF_2D_PPFEM_L2_H1_Convergence_h_p2}. In the Tables~~\ref{table:HMHF_2D_PPFEM_L2_H1_Convergence_tau_bdf1}--\ref{table:HMHF_2D_PPFEM_L2_H1_Convergence_tau_bdf2} we consider quadratic finite elements ($V_h^2$) with a fixed sufficiently small $h$ and vary $\tau$ to study the accuracy of the time discretization scheme.
\begin{table}[ht!]
\centering
\begin{tabular}{ |p{2.5cm}|| p{2.65cm}| p{1cm}|| p{2.65cm}|p{1cm}|}
 \hline
  $h=2^{-6}$ & $\lVert \mathbf{u}_{h}^{J} - \mathbf{u}^{\text{ref},J}_{h} \rVert_{L^2}$ & EOC & $\lVert \mathbf{u}_{h}^{J} - \mathbf{u}^{\text{ref},J}_{h}  \rVert_{H^1}$ & EOC\\
 \hline
 $\tau=5 \cdot 10^{-2}$  & $3.4910e-02$ & $-$  & $1.4780e-01$ & $-$ \\ 
 $\tau=2.5 \cdot 10^{-2}$  & $2.3253e-02$ & $0.59$  & $9.6176e-02$ & $0.62$ \\
 $\tau=1.25 \cdot 10^{-2}$ & $1.3985e-02$ & $0.73$ & $5.7423e-02$ & $0.74$\\
 $\tau=6.25 \cdot 10^{-3}$ & $7.8473e-03$ & $0.83$ & $3.2146e-02$ & $0.84$\\
$\tau=3.125 \cdot 10^{-3}$ & $4.2028e-03$ & $0.90$  & $1.7204e-02$& $0.90$ \\
$\tau=1.5625 \cdot 10^{-3}$ & $2.1849e-03$ & $0.94$ & $8.9428e-03$ & $0.94$\\
$\tau=7.8125 \cdot 10^{-4}$ & $1.1158e-03$ & $0.97$ & $4.5677e-03$ & $0.97$\\
 \hline 
\end{tabular}
\caption{Error for PPFEM \eqref{eq:HMHF_PPFE}-\eqref{eq:HMHF_PPFE_normalization} with BDF1 and $\mathbf{{V}}_{h}^2$; $h=2^{-6}$ fixed. }
\label{table:HMHF_2D_PPFEM_L2_H1_Convergence_tau_bdf1}
\end{table}
\begin{table}[ht!]
\centering
\begin{tabular}{ |p{2.5cm}|| p{2.65cm}| p{1cm}|| p{2.65cm}|p{1cm}|}
 \hline
  $h=2^{-6}$ & $\lVert \mathbf{u}_{h}^{J} - \mathbf{u}^{\text{ref},J}_{h} \rVert_{L^2}$ & EOC & $\lVert \mathbf{u}_{h}^{J} - \mathbf{u}^{\text{ref},J}_{h}  \rVert_{H^1}$ & EOC\\
 \hline
  $\tau=5 \cdot 10^{-2}$  & $2.5378e-02$ & $-$  & $1.0438e-01$ & $-$ \\
 $\tau=2.5 \cdot 10^{-2}$  & $7.2493e-03$ & $1.81$  & $2.9959e-02$ & $1.80$ \\
 $\tau=1.25 \cdot 10^{-2}$ & $1.6907e-03$ & $2.10$ & $6.9553e-03$ & $2.11$\\
 $\tau=6.25 \cdot 10^{-3}$ & $4.1608e-04$ & $2.02$ & $1.7098e-03$ & $2.02$\\
$\tau=3.125 \cdot 10^{-3}$ & $1.0539e-04$ & $1.98$  & $4.3477e-04$ & $1.98$ \\
 \hline 
\end{tabular}
\caption{Error for PPFEM \eqref{eq:HMHF_PPFE}-\eqref{eq:HMHF_PPFE_normalization} with BDF2 and $\mathbf{{V}}_{h}^2$; $h=2^{-6}$ fixed. }
\label{table:HMHF_2D_PPFEM_L2_H1_Convergence_tau_bdf2}
\end{table}
Table~\ref{table:HMHF_2D_PPFEM_L2_H1_Convergence_tau_bdf1} shows that for a small enough time step we reach a rate of convergence  with optimal order 1 for BDF1 with respect to the time step $\tau$ in $L^2$ and $H^1$ norms.  The results in Table~\ref{table:HMHF_2D_PPFEM_L2_H1_Convergence_tau_bdf2} show  the optimal convergence rate of order 2 for BDF2 with respect to $\tau$ in $L^2$ and $H^1$ norms. \\
In the Tables~~\ref{table:HMHF_2D_PPFEM_L2_H1_Convergence_h_p1}--\ref{table:HMHF_2D_PPFEM_L2_H1_Convergence_h_p2} we consider BDF2 with  a fixed sufficiently small $\tau$ and vary $h$ to study the accuracy of the finite element discretization.
\begin{table}[ht!]
\centering
\begin{tabular}{ |p{2.5cm}|| p{2.65cm}| p{1cm}|| p{2.65cm}|p{1cm}|}
 \hline
  $\tau = 10^{-6}$ & $\lVert \mathbf{u}_{h}^{J} - \mathbf{u}^{\text{ref},J}_{h} \rVert_{L^2}$ & EOC & $\lVert \mathbf{u}_{h}^{J} - \mathbf{u}^{\text{ref},J}_{h}  \rVert_{H^1}$ & EOC\\
 \hline
  $h=2^{-2}$  & $3.2225e-02$ & $-$  & $1.8987e-01$ & $-$ \\
 $h=2^{-3}$  & $8.1866e-03$ & $1.98$  & $8.5117e-02$ & $1.16$ \\
 $h=2^{-4}$ & $1.8507e-03$ & $2.15$ & $3.4665e-02$ & $1.30$\\
 $h=2^{-5}$ & $4.2263e-04$ & $2.13$ & $1.2632e-02$ & $1.45$\\
$h=2^{-6}$ & $1.0192e-04$ & $2.05$  & $4.5444e-03$ & $1.47$ \\
 \hline 
\end{tabular}
\caption{Error for PPFEM \eqref{eq:HMHF_PPFE}-\eqref{eq:HMHF_PPFE_normalization} with $\mathbf{{V}}_{h}^1$ and BDF2; $\tau= 10^{-6}$ fixed.}
\label{table:HMHF_2D_PPFEM_L2_H1_Convergence_h_p1}
\end{table}
\begin{table}[ht!]
\centering
\begin{tabular}{ |p{2.5cm}|| p{2.65cm}| p{1cm}|| p{2.65cm}|p{1cm}|}
 \hline
  $\tau = 10^{-6}$ & $\lVert \mathbf{u}_{h}^{J} - \mathbf{u}^{\text{ref},J}_{h} \rVert_{L^2(I)}$ & EOC & $\lVert \mathbf{u}_{h}^{J} - \mathbf{u}^{\text{ref},J}_{h}  \rVert_{H^1(I)}$ & EOC\\
 \hline
 $h=2^{-2}$  & $8.2075e-02$ & $-$  & $3.4101e-01$ & $-$ \\
 $h=2^{-3}$ & $3.6998e-02$ & $1.15$ & $1.5203e-01$ & $1.17$\\
 $h=2^{-4}$ & $2.5116e-03$ & $3.88$ & $1.0399e-02$ & $3.87$\\
 $h=2^{-5}$ & $1.6563e-04$ & $3.92$  & $6.9987e-04$ & $3.89$ \\
 $h=2^{-6}$ & $1.1221e-05$ & $3.88$  & $4.8919e-05$ & $3.84$ \\
 \hline 
\end{tabular}
\caption{Error for PPFEM \eqref{eq:HMHF_PPFE}-\eqref{eq:HMHF_PPFE_normalization} with $\mathbf{{V}}_{h}^2$ and BDF2; $ \tau= 10^{-6}$ fixed. }
\label{table:HMHF_2D_PPFEM_L2_H1_Convergence_h_p2}
\end{table}

\begin{table}[ht!]
	\centering
	\begin{tabular}{ |p{2.5cm}|| p{2.65cm}| p{1cm}|| p{2.65cm}|p{1cm}|}
		\hline
		$\tau = 10^{-6}$ & $\lVert \mathbf{u}_{h}^{J} - \mathbf{u}^{\text{ref},J}_{h} \rVert_{L^2(I)}$ & EOC & $\lVert \mathbf{u}_{h}^{J} - \mathbf{u}^{\text{ref},J}_{h}  \rVert_{H^1(I)}$ & EOC\\
		\hline
		$h=2^{-2}$  & $8.8618e-02$ & $-$  & $4.2909e-01$ & $-$ \\
		$h=2^{-3}$ & $8.8069e-02$ & $0.14$ & $2.8156e-01$ & $0.61$\\
		$h=2^{-4}$ & $9.6222e-03$ & $3.07$ & $3.4468e-02$ & $3.03$\\
		$h=2^{-5}$ & $7.0798e-04$ & $3.76$  & $2.7698e-03$ & $3.64$ \\
		$h=2^{-6}$ & $7.5579e-05$ & $3.23$  & $3.1968e-04$ & $3.12$ \\
		\hline 
	\end{tabular}
	\caption{ Error for PPFEM \eqref{eq:HMHF_PPFE}-\eqref{eq:HMHF_PPFE_normalization} with $\mathbf{{V}}_{h}^2$ and BDF2 and initial condition $u_0(r)=0.5\pi\left(\sin(2\pi r)+r\right)$; $ \tau= 10^{-6}$ fixed }
	\label{table:HMHF_2D_PPFEM_L2_H1_Convergence_h_p2_experiment_2}
\end{table}

Table~\ref{table:HMHF_2D_PPFEM_L2_H1_Convergence_h_p1} shows optimal order of convergence in the $L^2$ norm with respect to $h$, while convergence rate in the $H^1$ norm is slightly above the optimal rate of 1 for the piecewise linear finite element space.
Table~\ref{table:HMHF_2D_PPFEM_L2_H1_Convergence_h_p2} shows a significantly higher convergence rate than expected for quadratic finite elements. Furthermore,  the error reductions  in the $H^1$-norm and in the $L^2$-norm are approximately the same. Repeating the numerical experiment with a different initial condition $u_0(r)=0.5\pi\left(\sin(2\pi r)+r\right)$ we observe an $L^2$ norm convergence with an order closer to the optimal order 3, cf.   Table~\ref{table:HMHF_2D_PPFEM_L2_H1_Convergence_h_p2_experiment_2} (third column). However, also in this case we  observe a (unexpected) similar reduction rate in the $H^1$ norm. We have no satisfactory explanation for these observations. 

\subsubsection{Treatment of unit length constraint using a tangent space approach}
From 
\begin{align*}
  \frac{1}{2} \partial_t \left(\left| \mathbf{u} \right|^2\right) = \partial_t\mathbf{u} \cdot \mathbf{u} 
\end{align*}
it follows that  if the initial condition $\mathbf{u}(0,\cdot)=\mathbf{u}_0$ satisfies $|\mathbf{u}_0|=1$ on $D$ then the solution $\mathbf{u}$ satisfies the unit length constraint $|\mathbf{u}(t,\cdot)|=1$ on $D$ for $t \in [0,T]$ iff $\partial_t\mathbf{u} \cdot \mathbf{u} =0$ on $[0,T] \times D$. 
In  \cite{Alouges2008FEMLL,AlougesJaisson06,KovacsLuebich}, this elementary observation is used  to construct a discretization method for the Landau-Lifshitz-Gilbert equation. In \cite{Alouges2008FEMLL,AlougesJaisson06}, a discrete tangent space is constructed where the orthogonality is satisfied at each mesh point. Here, we apply the method studied in \cite{KovacsLuebich} to  the harmonic map heat flow problem, where an $L^2$ projection of the orthogonality condition onto the finite element space is used, cf. \eqref{constraint}.
\\
Let
\begin{align}
   \mathbf{T}(\mathbf{u}) := \{ \mathbf{v} \in L^2({D};\mathbb{R}^3)] : \mathbf{u} \cdot \mathbf{v} = 0 \text{ a.e.} \}  
\end{align}
be the tangent space at some $\mathbf{u}\in H^1({D},\mathbb{R}^3)$. Note that if $\mathbf{u}$ is the solution of \eqref{HMHFeq}, then for $\mathbf{v} \in \mathbf{T}(\mathbf{u})$ we have $|\nabla \mathbf{u}|^2 \mathbf{u}\cdot\mathbf{v}=0$. Using this and the elementary observation above leads to the following variational formulation.  Given $\mathbf{u}_0 \in H^1({D};\mathbb{R}^3) $ with $\left|\mathbf{u}_0\right|=1$ on $D$,  find $\mathbf{u} \in C^1\big([0,T];H^1({D},\mathbb{R}^3)\big)$  such that $\mathbf{u}(0,\cdot)=\mathbf{u}_0$ and  a time derivative  $
    \partial_t \mathbf{u}(t,\cdot) \in \mathbf{T}(\mathbf{u}(t,\cdot))
$ 
such that $(\partial_t \mathbf{u})_{|\partial D}=0$ and satisfies for all $\mathbf{v} \in \mathbf{T}(\mathbf{u}(t,\cdot))\cap H^1_0({D},\mathbb{R}^3)$
\begin{equation*}
    \left(\partial_t \mathbf{u}(t,\cdot), \mathbf{v} \right) + \left(\nabla \mathbf{u}(t,\cdot), \nabla \mathbf{v} \right) = 0, \quad t \in (0,T],
\end{equation*}
where $(\cdot,\cdot)$ is again the $L^2$ inner product on $D$.
In  the discretization, we use a discretized tangent space where the orthogonality is reduced to orthogonality in the finite element space as follows: 
\begin{align} \label{constraint}
    \mathbf{T}_{h,0}^p(\mathbf{u}) := \left\{ \mathbf{v}_{h} \in \mathbf{{V}}_{h,0}^p: \mathbf{b}(\mathbf{u};\mathbf{v}_{h},w_{h}) :=\left(\mathbf{u}\cdot \mathbf{v}_{h}, w_{h} \right) = 0 \; \forall w_{h} \in {V}_{h,0}^p \right\}.
\end{align}
We can write the approximation of the time derivative $\Dot{\mathbf{u}}_{h}^{j} \approx \partial_t \mathbf{u}(t_j,\cdot)$ by using \eqref{eq:HMHF_BDF_Def} for a given $k\in\{1,2\}$ as
\begin{align}
    \Dot{\mathbf{u}}_{h}^{j} = \frac{1}{\tau} \sum_{i=0}^{k} \delta_i^{(k)} \mathbf{u}_{h}^{j-i}. \label{eq:HMHF_Vector_BDF}
\end{align}
Since the tangent space depends itself on the solution of the current time step, we extrapolate the solution from previous time steps to construct the discrete tangent space. The extrapolation is given by, cf. \eqref{eq:HMHF_Approx_Extrapolation},
\begin{equation} \label{extrapol} \begin{split}
    \Hat{\mathbf{u}}^{j,(1)}_{h} &= \frac{ \mathbf{u}_{h}^{j-1}}{\left|\mathbf{u}_{h}^{j-1}\right|}, \\
    \Hat{\mathbf{u}}^{j,(2)}_{h} &= \frac{2 \mathbf{u}_{h}^{j-1}- \mathbf{u}_{h}^{j-2}}{\left|2\mathbf{u}_{h}^{j-1}-\mathbf{u}_{h}^{j-2}\right|}.
\end{split}
\end{equation}
Thus, we obtain the following scheme, cf. \cite[Equation (2.6)]{KovacsLuebich}. Let the time step $\tau$ and $J \in \Bbb{N}$ be such that $J \tau=T$.  Given $\tau >0$ and $p, k\in\{1,2\}$, $\mathbf{u}_{h}^{0}=\mathcal{I}_{h}(\mathbf{u}_0)$ and $\bu_h^1 \in \bV_{h, \bu_0}^p$ (only for $k=2$), for $j=k-1, \ldots, J-1$, find the approximate time derivative $\Dot{\mathbf{u}}_{h}^{j+1} \in \mathbf{T}^ p_{h,0}\left(\Hat{\mathbf{u}}_{h}^{j+1,(k)}\right)$ such that  for all $\mathbf{v}_{h} \in \mathbf{T}^p_{{h,0}}\left(\Hat{\mathbf{u}}_{h}^{j+1,(k)}\right)$
\begin{align}
    \left(\Dot{\mathbf{u}}_{h}^{j+1},\mathbf{v}_{h}\right) + \frac{\tau}{\delta^{(k)}_0} \left(\nabla \Dot{\mathbf{u}}_{h}^{j+1}, \nabla \mathbf{v}_{h} \right) = \sum_{i=1}^k \frac{\delta_i^{(k)}}{\delta_0^{(k)}} \left(\nabla \mathbf{u}_{h}^{j+1-i}, \nabla \mathbf{v}_{h} \right), \label{eq:HMHF_TFEM}
\end{align}
and the update step is given by
\begin{align}
    \mathbf{u}_{h}^{j+1} = \frac{\tau}{\delta_0^{(k)}} \Dot{\mathbf{u}}_{h}^{j+1} -\sum_{i=1}^k \frac{\delta_i^{(k)}}{\delta_0^{(k)}} \mathbf{u}_{h}^{j+1-i}, \label{eq:HMHF_TFEM_update}
\end{align}
which follows directly from \eqref{eq:HMHF_Vector_BDF}. For $k=2$ the initial function $\mathbf{u}_{h}^{1}$ is determined by applying one step of the BDF1  variant of the scheme. The scheme \eqref{eq:HMHF_TFEM}-\eqref{eq:HMHF_TFEM_update} is analyzed in \cite{KovacsLuebich}.

\begin{remark} \rm 
    In \cite[Theorem 3.1]{KovacsLuebich}, for  a variant of this scheme, modified for the LLG equation with Neumann boundary condition, an error bound as in \eqref{boundLubich} is derived (under a mild CFL-type condition) for BDF$k$ with $k\in \{1,2\}$. For  BDF$k$ with $3 \leq k \leq 5$, optimal $H^1$ norm error bounds of order  $\tau^k + h^p$ are derived for  $p \geq 2$, using a much stronger CFL type condition $\tau = \mathcal{O}(h)$ and with an additional assumption concerning the size of the damping parameter in the LLG equation, cf.~ \cite[Theorem 3.2]{KovacsLuebich}. The paper also contains results of numerical results with exact solutions for the LLG equation on a three-dimensional cube with a thin layer in $z$-direction to confirm the theoretical discretization error bounds.
\end{remark}
\ \\

Concerning implementation of this scheme we note the following. One can determine a basis of the finite element space $\mathbf{T}_{h,0}^p(\mathbf{u})$ by solving suitable local problems, cf. \cite[Section 2.2]{KovacsLuebich}. From  an implementation point of view this is inconvenient. Furthermore, it is not a-priori clear what the conditioning properties of this basis are.  An alternative is  a Lagrange multiplier approach, as also proposed in \cite[Section 2.2]{KovacsLuebich}. We introduce a Lagrange multiplier $\lambda_h \in V_{h,0}^p$ to treat  the constraint $\mathbf{u}_h \cdot \mathbf{v}_h=0$ in \eqref{constraint}. 
This leads to the following saddle point problem:
Given a time step $\tau >0$ (with $\tau J=T$) and $p, k \in \{1,2\}$, $\mathbf{u}_{h}^{0}=\mathcal{I}_{h}(\mathbf{u}_0)$, $\mathbf{u}_{h}^{1}\in \bV_{h, \bu_0}^p$ (only for $k=2$), for $j=k-1, \ldots, J-1$, find $\left(\Dot{\mathbf{u}}_{h}^{j+1}, \lambda_{h} \right) \in \left(\mathbf{{V}}_{h,0}^p,{V}_{h,0}^p\right)$ such that for all $\left(\mathbf{v}_{h},w_{h} \right) \in \left(\mathbf{{V}}_{h,0}^p, {V}_{h,0}^p \right) $
\begin{align}
    \left(\Dot{\mathbf{u}}_{h}^{j+1},\mathbf{v}_{h}\right) + \frac{\tau}{\delta_0^{(k)}} \left(\nabla \Dot{\mathbf{u}}_{h}^{j+1}, \nabla \mathbf{v}_{h} \right) &+  \mathbf{b}(\hat{\mathbf{u}}_{h}^{j+1,(k)}; \mathbf{v}_{h},\lambda_{h}) \notag \\
    &= \sum_{i=1}^k \frac{\delta_i^{(k)}}{\delta_0^{(k)}} \left(\nabla  \mathbf{u}_{h}^{j+1-i}, \nabla \mathbf{v}_{h}\right), \label{eq:HMHF_TFEM_1}\\
     \mathbf{b}(\hat{\mathbf{u}}_{h}^{j+1,(k)}; \Dot{\mathbf{u}}_{h}^{j+1},w_{h}) &= 0. \label{eq:HMHF_TFEM_2}
\end{align}
Given $\Dot{\mathbf{u}}_{h}^{j+1}$ we determine $\bu_{h}^{j+1}$ as in \eqref{eq:HMHF_TFEM_update}.
We call \eqref{eq:HMHF_TFEM_1}-\eqref{eq:HMHF_TFEM_2} together with \eqref{eq:HMHF_TFEM_update} the tangential finite element method (TFEM) for HMHF.

Concerning well-posedness of \eqref{eq:HMHF_TFEM_1}-\eqref{eq:HMHF_TFEM_2} and the relation between this scheme and \eqref{eq:HMHF_TFEM} we note the following. Let
$\left(\Dot{\mathbf{u}}_{h}^{j+1}, \lambda_{h} \right) \in \left(\mathbf{{V}}_{h,0}^p,{V}_{h,0}^p\right)$be a solution of \eqref{eq:HMHF_TFEM_1}-\eqref{eq:HMHF_TFEM_2};  then \eqref{eq:HMHF_TFEM_2} implies $\Dot{\mathbf{u}}_{h}^{j+1} \in \mathbf{T}^ p_{h,0}\left(\Hat{\mathbf{u}}_{h}^{j+1,(k)}\right)$ and, restricting in \eqref{eq:HMHF_TFEM_1} to $\mathbf{v}_h \in  \mathbf{T}^ p_{h,0}\left(\Hat{\mathbf{u}}_{h}^{j+1,(k)}\right)$, then yields that $\Dot{\mathbf{u}}_{h}^{j+1}$ solves \eqref{eq:HMHF_TFEM}.  Hence, if the discrete saddle point problem \eqref{eq:HMHF_TFEM_1}-\eqref{eq:HMHF_TFEM_2}  has a unique solution, then its $\mathbf{u}$-solution is the same as that of \eqref{eq:HMHF_TFEM}. 

In \cite{KovacsLuebich} the well-posedness of \eqref{eq:HMHF_TFEM_1}-\eqref{eq:HMHF_TFEM_2} is not addressed. In the subsection below we show that under reasonable assumptions the saddle point formulation is well-posed.
\\[1ex]
\paragraph*{Well-posedness of the saddle point formulation}
The bilinear form $(\mathbf{w}_h,\mathbf{v}_h) \to \left(\mathbf{w}_{h},\mathbf{v}_{h}\right) + \frac{\tau}{\delta_0^{(k)}} \left(\nabla \mathbf{w}_{h}^{j+1}, \nabla \mathbf{v}_{h} \right)$ used in \eqref{eq:HMHF_TFEM_1} is elliptic on $\mathbf{{V}}_{h,0}^p$. 
Hence, we have well-posedness of \eqref{eq:HMHF_TFEM_1}-\eqref{eq:HMHF_TFEM_2} iff the bilinear form $\mathbf{b}(\hat{\mathbf{u}}_{h}^{j+1,(k)}; \cdot, \cdot)$ has the discrete  inf-sup property.  We have the following result: 
\begin{lemma} \label{leminfsup}
Assume that the solution $\bu$ of \eqref{HMHFeq} is sufficiently smooth and for the discrete approximation $(\bu_h^j)_{0 \leq j \leq J}$ of \eqref{eq:HMHF_TFEM}-\eqref{eq:HMHF_TFEM_update} we have uniform convergence $\bu_h^j  \to \bu(t_j,\cdot)$ in the following sense: For arbitrary $\varepsilon \in (0,1]$ there exist $\varepsilon_1 >0, \varepsilon_2 > 0$ (depending on $\bu$) such that
\begin{equation} \label{condconv}
  \delta_{\tau, h}:= \max_{0 \leq j \leq J} \|\bu_h^j - \bu(t_j,\cdot)\|_{L^\infty} \leq \varepsilon \quad \text{for all}~ \tau \leq \varepsilon_1,~h \leq \varepsilon_2.
\end{equation}
Then there exists $\beta=\beta(\bu) >0$, independent of $h$ and $\tau$, such that for $k\in \{0,1\}$ and  $\tau$, $h$ sufficiently small we have:
\begin{equation}\label{discrinfs}
    \sup_{\mathbf{v}_{h} \in \mathbf{{V}}_{h,0}^p} \frac{\mathbf{b}(\hat{\mathbf{u}}_{h}^{j+1,(k)}; \mathbf{v}_{h},w_{h})}{\|\bv_h\|_{L^2}} \geq \beta\,  \|w_h\|_{L^2} \quad \text{for all}~w_h \in V_{h,0}^p, ~0 \leq j \leq J-1.
\end{equation}
\end{lemma}
\begin{proof} A proof is given in the Appendix \ref{appendix:InfSupStab}.
 \end{proof}
 \ \\
 The result \eqref{discrinfs} yields well-posedness of \eqref{eq:HMHF_TFEM_1}-\eqref{eq:HMHF_TFEM_2}, provided assumption \eqref{condconv} is satisfied and $\tau $ and $h$ are sufficiently small. Moreover, the discrete inf-sup constant $\beta$ in \eqref{discrinfs} is independent of the discretization parameters $h$ and $\tau$, which implies a control of the condition number of the Schur complement of the symmetric indefinite system matrix corresponding to the saddle point problem \eqref{eq:HMHF_TFEM_1}-\eqref{eq:HMHF_TFEM_2}.
 \begin{remark} \label{reminfsup} \rm 
 Concerning the assumption \eqref{condconv} we note the following. In \cite[Theorem 3.1]{KovacsLuebich}, for a variant of the  scheme \eqref{eq:HMHF_TFEM}-\eqref{eq:HMHF_TFEM_update} applied to the LLG equation, an optimal error bound of the form
 \begin{equation} \label{boundLubich}
    \max_{0 \leq j \leq J} \|\bu_h^j - \bu(t_j,\cdot)\|_{H^1} \leq c( \tau^k +h^p), \quad k=1,2,
 \end{equation}
is derived under the very mild CFL-type condition $\tau^k \leq \bar c h^\frac12$ for a sufficiently small constant $\bar c$. Assume that \eqref{boundLubich} also holds for the scheme applied to the HMHF problem and assume that a slightly stronger mesh-parameter condition is satisfied, namely $\tau^k \leq c h^{\frac12 +\epsilon}$ holds for some constants $c > 0$ and $\epsilon >0$. Then we obtain, with $I_h$ the nodal interpolation in the finite element space $\bV_{h,0}^p$:
\begin{align*}
 &  \max_{0 \leq j \leq J} \|\bu_h^j - \bu(t_j,\cdot)\|_{L^\infty} \leq \max_{0 \leq j \leq J} \|\bu_h^j - I_h \bu(t_j,\cdot)\|_{L^\infty} + c h^{p+1} \\
 & \leq c h^{-\frac12} \max_{0 \leq j \leq J} \|\bu_h^j - I_h \bu(t_j,\cdot)\|_{H^1} + c h^{p+1}  \leq c h^{-\frac12}\max_{0 \leq j \leq J} \|\bu_h^j -  \bu(t_j,\cdot)\|_{H^1} + c h^{p - \frac12} \\ &  \leq c h^{-\frac12} \tau^k +c h^{p - \frac12} \leq c h^{\min \{\epsilon, p -\frac12\} },  
\end{align*}
which implies that the assumption \eqref{condconv} is satisfied. 
 \end{remark}
\paragraph*{Numerical results}
A reference solution $\mathbf{u}_{h}^{\text{ref}}$ is determined in the same way as in Section \ref{sec:HMHF_PPFEM}. The discrete approximations $\mathbf{u}_{h}^j$ are computed using TFEM \eqref{eq:HMHF_TFEM_1}-\eqref{eq:HMHF_TFEM_2} and \eqref{eq:HMHF_TFEM_update}. Results are shown in Tables~\ref{table:HMHF2D_TFEM_L2_H1_Convergence_tau_bdf1}--\ref{table:HMHF_2D_TFEM_L2_H1_Convergence_h_p2}.\\
\begin{table}[ht!]
\centering
\begin{tabular}{ |p{2.5cm}|| p{2.65cm}| p{1cm}|| p{2.65cm}|p{1cm}|}
 \hline
  $h=2^{-6}$ & $\lVert \mathbf{u}_{h}^{J} - \mathbf{u}^{\text{ref},J}_{h} \rVert_{L^2}$ & EOC & $\lVert \mathbf{u}_{h}^{J} - \mathbf{u}^{\text{ref},J}_{h}  \rVert_{H^1}$ & EOC\\
 \hline
 $\tau=5 \cdot 10^{-2}$  & $1.0814e-01$ & $-$  & $4.3607e-01$ & $-$ \\ 
 $\tau=2.5 \cdot 10^{-2}$  & $6.1103e-02$ & $ 0.82$  & $2.4512e-01$ & $0.83$ \\
 $\tau=1.25 \cdot 10^{-2}$ & $3.2653e-02$ & $0.90$ & $1.3072e-01$ & $0.91$\\
 $\tau=6.25 \cdot 10^{-3}$ & $1.6902e-02$ & $0.95$ & $6.7616e-02$ & $0.95$\\
$\tau=3.125 \cdot 10^{-3}$ & $8.6006e-03$ & $0.97$  & $3.4404e-02$& $0.97$ \\
$\tau=1.5625 \cdot 10^{-3}$ & $4.3381e-03$ & $0.99$ & $1.7355e-02$ & $0.99$\\
$\tau=7.8125 \cdot 10^{-4}$ & $2.1783e-03$ & $0.99$ & $8.7161e-03$ & $0.99$\\

 \hline 
\end{tabular}
\caption{Error for TFEM (\eqref{eq:HMHF_TFEM_1}-\eqref{eq:HMHF_TFEM_2} and \eqref{eq:HMHF_TFEM_update}) with BDF1 and $\mathbf{{V}}_{h}^2$;  $h=2^{-6}$ fixed. }
\label{table:HMHF2D_TFEM_L2_H1_Convergence_tau_bdf1}
\end{table}
\begin{table}[ht!]
\centering
\begin{tabular}{ |p{2.5cm}|| p{2.65cm}| p{1cm}|| p{2.65cm}|p{1cm}|}
 \hline
  $h=2^{-6}$ & $\lVert \mathbf{u}_{h}^{J} - \mathbf{u}^{\text{ref},J}_{h} \rVert_{L^2}$ & EOC & $\lVert \mathbf{u}_{h}^{J} - \mathbf{u}^{\text{ref},J}_{h}  \rVert_{H^1}$ & EOC\\
 \hline
 $\tau=5 \cdot 10^{-2}$  & $7.4457e-02$ & $-$  & $2.9762e-01$ & $-$ \\ 
 $\tau=2.5 \cdot 10^{-2}$  & $2.0902e-02$ & $ 1.83$  & $8.4048e-02$ & $1.82$ \\
 $\tau=1.25 \cdot 10^{-2}$ & $5.2568e-03$ & $1.99$ & $2.2060e-02$ & $1.93$\\
 $\tau=6.25 \cdot 10^{-3}$ & $1.3969e-03$ & $1.91$ & $6.3743e-03$ & $ 1.80$\\
$\tau=3.125 \cdot 10^{-3}$ & $3.6834e-04$ & $1.92$  & $1.8829e-03$& $1.76$ \\
 \hline 
\end{tabular}
\caption{Error for TFEM (\eqref{eq:HMHF_TFEM_1}-\eqref{eq:HMHF_TFEM_2} and \eqref{eq:HMHF_TFEM_update}) with BDF2  and $\mathbf{{V}}_{h}^2$; $h=2^{-6}$ fixed. }
\label{table:HMHF2D_TFEM_L2_H1_Convergence_tau_bdf2}
\end{table}

Similar to PPFEM, Table \ref{table:HMHF2D_TFEM_L2_H1_Convergence_tau_bdf1} shows  for BDF1  an optimal convergence rate of order 1  in $\tau$, both in the $L^2$- and $H^1$-norm. 
In Table~\ref{table:HMHF2D_TFEM_L2_H1_Convergence_tau_bdf2} we observe the optimal convergence rate of order 2 for BDF2 with respect to $\tau$ in $L^2$ and slightly less than 2 in the $H^1$-norm. 
\begin{table}[ht!]
\centering
\begin{tabular}{ |p{2.5cm}|| p{2.65cm}| p{1cm}|| p{2.65cm}|p{1cm}|}
 \hline
  $\tau = 10^{-6}$ & $\lVert \mathbf{u}_{h}^{J} - \mathbf{u}^{\text{ref},J}_{h} \rVert_{L^2}$ & EOC & $\lVert \mathbf{u}_{h}^{J} - \mathbf{u}^{\text{ref},J}_{h}  \rVert_{H^1}$ & EOC\\
 \hline
  $h=2^{-2}$  & $3.7887e-02$ & $-$  & $2.0071e-01$ & $-$ \\
 $h=2^{-3}$  & $9.2785e-03$ & $2.03$  & $8.7882e-02$ & $1.19$ \\
 $h=2^{-4}$ & $2.0418e-03$ & $2.18$ & $3.5116e-02$ & $1.32$\\
 $h=2^{-5}$ & $4.6612e-04$ & $2.13$ & $1.2721e-02$ & $1.46$\\
$h=2^{-6}$ & $1.1216e-04$ & $2.06$  & $4.5600e-03$ & $1.48$ \\
 \hline 
\end{tabular}
\caption{Error for TFEM (\eqref{eq:HMHF_TFEM_1}-\eqref{eq:HMHF_TFEM_2} and \eqref{eq:HMHF_TFEM_update}) with $\mathbf{{V}}_{h}^1$ and BDF2; $\tau = 10^{-6}$ fixed.}
\label{table:HMHF_2D_TFEM_L2_H1_Convergence_h_p1}
\end{table}
\begin{table}[ht!]
\centering
\begin{tabular}{ |p{2.5cm}|| p{2.65cm}| p{1cm}|| p{2.65cm}|p{1cm}|}
 \hline
  $\tau = 10^{-6}$ & $\lVert \mathbf{u}_{h}^{J} - \mathbf{u}^{\text{ref},J}_{h} \rVert_{L^2}$ & EOC & $\lVert \mathbf{u}_{h}^{J} - \mathbf{u}^{\text{ref},J}_{h}  \rVert_{H^1}$ & EOC\\
 \hline
  $h=2^{-2}$  & $5.5044e-03$ & $-$  & $4.4757e-02$ & $-$ \\
 $h=2^{-3}$  & $6.0176e-04$ & $3.19$  & $7.0017e-03$ & $2.68$ \\
 $h=2^{-4}$ & $7.0011e-05$ & $3.10$ & $1.1024e-03$ & $2.67$\\
 $h=2^{-5}$ & $ 8.3761e-06$ & $3.06$ & $1.8427e-04$ & $ 2.58$\\
  $h=2^{-6}$ & $ 1.0403e-06$ & $3.01$ & $3.4234e-05$ & $ 2.43$\\
 \hline 
\end{tabular}
\caption{Error for TFEM (\eqref{eq:HMHF_TFEM_1}-\eqref{eq:HMHF_TFEM_2} and \eqref{eq:HMHF_TFEM_update}) with $\mathbf{{V}}_{h}^2$ and BDF2; $\tau = 10^{-6}$ fixed. }
\label{table:HMHF_2D_TFEM_L2_H1_Convergence_h_p2}
\end{table}

Table~\ref{table:HMHF_2D_TFEM_L2_H1_Convergence_h_p1} shows an optimal convergence rate of order 2 in the $L^2$ norm with respect to $h$, while the convergence rate in the $H^1$-norm is slightly above the optimal rate of 1, again similar to PPFEM, cf. Table~\ref{table:HMHF_2D_PPFEM_L2_H1_Convergence_h_p1}. Table~\ref{table:HMHF_2D_TFEM_L2_H1_Convergence_h_p2} also shows an optimal convergence rate of order 3 in the $L^2$ norm with respect to $h$, while the convergence rate in the $H^1$-norm is slightly above the optimal rate of 2. \\
Summarizing, the results for PPFEM and TFEM are very similar, except for BDF2 combined with  $\mathbf{V}_{h}^2$ in the Tables~\ref{table:HMHF_2D_PPFEM_L2_H1_Convergence_h_p2} and \ref{table:HMHF_2D_TFEM_L2_H1_Convergence_h_p2}. 


\subsubsection{Treatment of unit length constraint based on a constraint preserving formulation}
\label{sec:HMHF_Bartels_Method}
In view of the vector identity
\begin{align*}
    \mathbf{a} \times (\mathbf{b} \times  \mathbf{c}) = \left(\mathbf{a}\cdot \mathbf{c} \right) \mathbf{b} - \left(\mathbf{a}\cdot \mathbf{b} \right) \mathbf{c} \quad \forall \mathbf{a},\mathbf{b},\mathbf{c} \in \mathbb{R}^3,
\end{align*}
we can write  
\begin{align*}
     \mathbf{u} \times (\mathbf{u} \times  \Laplace \mathbf{u}) = \left(\mathbf{u} \cdot \Laplace \mathbf{u} \right) \mathbf{u} - |\mathbf{u}|^2 \Laplace \mathbf{u}.
\end{align*}
Because of $\mathbf{u} \in \mathbb{S}^2$, we have $|\mathbf{u}|=1$ and $\mathbf{u}\cdot \nabla \mathbf{u} = 0$ and obtain
\begin{align*}
    \mathbf{u} \times (\mathbf{u} \times  \Laplace \mathbf{u}) &= \left(\mathbf{u} \cdot \Laplace \mathbf{u} \right) \mathbf{u} - \Laplace \mathbf{u} \\
    &= \left(\nabla \left(\mathbf{u}\cdot \nabla \mathbf{u}\right) - \nabla \mathbf{u} \cdot \nabla \mathbf{u} \right) \mathbf{u}-\Laplace \mathbf{u} \\
    &= - \left| \nabla \mathbf{u} \right|^2 \mathbf{u} - \Laplace \mathbf{u}.
\end{align*}
Thus, we can rewrite \eqref{HMHFeq} as
\begin{align}
    \partial_t\mathbf{u}= - \mathbf{u} \times (\mathbf{u} \times  \Laplace \mathbf{u}) \text{ on } D, t \in (0,T]. \label{eq:HMHF_PDE_DoubleCross}
\end{align}
Note that multiplying both sides of \eqref{eq:HMHF_PDE_DoubleCross} with $\mathbf{u}$ results in
$
    \partial_t\mathbf{u} \cdot \mathbf{u} = 0 
$,
and thus
\begin{align} \label{preserv}
    \frac{1}{2} \partial_t \left(\left| \mathbf{u} \right|^2\right)=0.
\end{align}
Since $|\bu_0|=1$ on $D$ holds, this implies that a solution of \eqref{eq:HMHF_PDE_DoubleCross} preserves the unit length property $\left| \mathbf{u} \right|=1$ on $D$. 
The method presented in \cite{BartelsProhl2007} is based on the double cross formulation of the harmonic map heat flow \eqref{eq:HMHF_PDE_DoubleCross}. A variational formulation reads: Given $\mathbf{u}_0 \in H^1({D};\mathbb{R}^3) $ with $\left|\mathbf{u}_0\right|=1$ (a.e.), find $\mathbf{u} \in C^1([0,T]; H^1({D};\mathbb{R}^3)$  with $\mathbf{u}(0,\cdot)=\mathbf{u}_0$ and $\bu_{|\partial D}=(\bu_0)_{|\partial D}$ such that for all $\bv \in H^1_0({D};\mathbb{R}^3)$
\begin{align*}
    &\left(\partial_t \mathbf{u}, \mathbf{v}\right) + \left(\mathbf{u} \times (\mathbf{u} \times  \Laplace \mathbf{u}), \mathbf{v} \right) = 0, \quad t \in (0,T].
\end{align*}
Let $\{ \phi_z: z\in \mathbf{\mathcal{N}}_{h}\} $ be the nodal basis of the linear finite element space $V_{h}^1$. For  $\mathbf{u},\mathbf{v} \in C^0(\Bar{D};\mathbb{R}^3)$, we define 
\begin{align*}
    \left(\mathbf{u},\mathbf{v}\right)_{h} :=\sum_{z\in \mathbf{\mathcal{N}}_{h}} \beta_z \langle \mathbf{u}(z),\mathbf{v}(z) \rangle, \quad \beta_z := \int_{D}\phi_z dx,
\end{align*}
where $\langle \cdot, \cdot \rangle$ is the Euclidean scalar product in $\mathbb{R}^3$. The corresponding seminorm is denoted by $\lVert \mathbf{u} \rVert_{h} :=(\mathbf{u},\mathbf{u})_{h}^\frac12 $. On $\bV_h^1$ this defines a norm that  is equivalent to the $L^2$-norm \cite[Equation (2.3)]{BartelsProhl2007}. The discrete Laplacian $\Tilde{\Laplace}_{h}:H^1(D;\mathbb{R}^3)\rightarrow \mathbf{V}_{h}^1$ is given by
\begin{align*}
    -\left(\Tilde{\Laplace}_{h} \mathbf{u},\mathbf{v} \right)_{h} = \left(\nabla \mathbf{u},\nabla \mathbf{v}\right) \quad \text{for all}~ \mathbf{v} \in \mathbf{V}_{h}^1.
\end{align*}
We define 
\begin{align*}
    \Bar{\mathbf{u}}_{h}^{j+1/2} := \frac{1}{2}\left(\mathbf{u}_{h}^{j+1}+\mathbf{u}_{h}^j\right). 
\end{align*}
Following \cite{BartelsProhl2007}, we use BDF1 for time discretization. For a given time step $\tau$ (with $\tau J=T$) the resulting discretization reads: Given $\mathbf{u}_{h}^{0}=\mathcal{I}_{h}(\mathbf{u}_0)$, for $j=0, \ldots, J-1$, find $\mathbf{u}_{h}^{j+1} \in \mathbf{V}_{h,\mathbf{u}_0}^1$ such that for all $\mathbf{v}_{h} \in \mathbf{V}^1_{h,0}$
\begin{align}
    \frac{1}{\tau}\left(\mathbf{u}^{j+1}_{h} - \mathbf{u}^j_{h}, \mathbf{v}_{h} \right)_{h} + \left(\Bar{\mathbf{u}}_{h}^{j+1/2}\times (\Bar{\mathbf{u}}_{h}^{j+1/2} \times  \Tilde{\Laplace}_{h} \Bar{\mathbf{u}}_{h}^{j+1/2}), \mathbf{v}_{h}\right)_{h} = 0. \label{eq:HMHF_BFEM}
\end{align}
Testing \eqref{eq:HMHF_BFEM} with $\mathbf{v}_{h}=\Bar{\mathbf{u}}_{h}^{j+1/2}(z)\mathbf{\phi}_z$ for $z \in {\mathbf{\mathcal{N}}_{h}}$ results in
\begin{align}
    \frac{1}{\tau} \left[ \left(\mathbf{u}^{j+1}_{h}(z)\right)^2 - \left(\mathbf{u}^{j}_{h}(z) \right)^2  \right] = 0, \label{eq:HMHF_2D_CrossFEM_UnitLength}
\end{align}
which is a discrete analogon of the length preservation property \eqref{preserv}. It implies that the constraint is pointwise satisfied at every mesh point $z \in {\mathbf{\mathcal{N}}_{h}}$ provided  the initial condition satisfies the norm constraint. 
Due to the cross product the discrete problem \eqref{eq:HMHF_BFEM} for the unknown $\mathbf{u}_{h}^{j+1}$ is strongly nonlinear. In  \cite{BartelsProhl2007},  the following iterative linearization method is proposed. 
We introduce $\mathbf{w}^{j+1}_{h}:=\Bar{\mathbf{u}}_{h}^{j+1/2}$ and rewrite \eqref{eq:HMHF_BFEM} as
\begin{align} \label{nonlineareq}
    \frac{2}{\tau}\left(\mathbf{w}^{j+1}_{h} , \mathbf{v}_{h} \right)_{h} + \left(\mathbf{w}_{h}^{j+1}\times (\mathbf{w}_{h}^{j+1} \times  \Tilde{\Laplace}_{h} \mathbf{w}_{h}^{j+1}), \mathbf{v}_{h}\right)_{h} = \frac{2}{\tau}\left(\mathbf{u}^j_{h},\mathbf{v}_{h}\right)_{h}
\end{align}
for all $\mathbf{v}_{h} \in \mathbf{V}^1_{h,0}$.
To this nonlinear problem a fixed point iteration (Algorithm \ref{algo:HMHF_Bartels}) is applied, cf.~\cite[Algorithm 4.1]{BartelsProhl2007}. 

\begin{algorithm}[ht!]
    \SetKwRepeat{Do}{do}{while}
    \caption{From \cite[Algorithm 4.1]{BartelsProhl2007}}
    \label{algo:HMHF_Bartels}
    \textbf{Given:} $\mathbf{u}_{h}^{0}=\mathcal{I}_{h}(\mathbf{u}_0), \tau >0, \epsilon>0$ and $j=0$ \\[1ex]
    \While{$ j \tau  \leq T$}
    {
        Set $\mathbf{w}_{h}^{j+1,0}:=\mathbf{u}_{h}^j$ and $l:=0$ \\
        \Do{$\lVert \mathbf{R}_{h}^{j+1} \rVert \geq \epsilon$}
        {
        Compute $\mathbf{w}_{h} \in \mathbf{V}^1_{h,\mathbf{u}_0}$ such that for all $\mathbf{v}_{h} \in \mathbf{V}^1_{h,0}$ 
        \begin{align*}
           \frac{2}{\tau}\big(\mathbf{w}_{h} , \mathbf{v}_{h} \big)_{h} + \big(\mathbf{w}_{h}\times (\mathbf{w}_{h}^{j+1,l} \times  \Tilde{\Laplace}_{h} \mathbf{w}_{h}^{j+1,l}), \mathbf{v}_{h}\big)_{h} = \frac{2}{\tau}\big(\mathbf{u}^j_{h},\mathbf{v}_{h}\big)_h
        \end{align*}
        Set $\mathbf{w}^{j+1,l+1}_{h}:=\mathbf{w}_h$, $\mathbf{e}_{h}^{j+1,l+1}:=\mathbf{w}^{j+1,l+1}_{h}-\mathbf{w}^{j+1,l}_{h}$ and compute 
        \begin{align*}
            \mathbf{R}_{h}^{j+1} = \mathbf{w}^{j+1,l+1}_{h} \times \Tilde{\Laplace}_{h} \mathbf{e}_{h}^{j+1,l+1} + \mathbf{e}_{h}^{j+1,l+1} \times \Tilde{\Laplace}_{h} \mathbf{w}_{h}^{j+1,l}
        \end{align*}
        Set $l:=l+1$
        }
        Set $\mathbf{u}_{h}^{j+1}:=2 \mathbf{w}^{j+1,l+1}_{h} - \mathbf{u}_{h}^{j}$ and $j:=j+1$
    }
\end{algorithm}

For the constraint preserving discretization \eqref{nonlineareq} combined with the fixed point linearization Algorithm \ref{algo:HMHF_Bartels} we introduce the notation CPFEM$+$FP.
\begin{remark} \rm 
	In \cite{BartelsProhl2007}, for the case of a Neumann boundary condition, a result on (unconditional) weak convergence (for $h \to 0$) of the solution of \eqref{eq:HMHF_BFEM} to the (weak) solution of the harmonic map heat flow \eqref{HMHFeq} is derived. No result for the rate of  convergence rate is available. 
	Furthermore, in \cite[Theorem 4.1]{BartelsProhl2007}, it is proved that Algorithm \ref{algo:HMHF_Bartels} converges provided the strong CFL condition $\tau = \mathcal{O}(h^2)$ is satisfied. No rigorous results on rate of convergence of Algorithm \ref{algo:HMHF_Bartels} are known. 
\end{remark}

In \cite{BartelsProhl2007} the fixed point linearization as in Algorithm \ref{algo:HMHF_Bartels} is used because it fits in the analysis framework of that paper. 
Instead of this fixed point linearization one can also use the Newton method, which, however, does not fit in the analysis framework of \cite{BartelsProhl2007}. We derive the Newton linearization and also include numerical results using this method. Note that for $\by, \bz \in \Bbb{R}^3$, $A \in \Bbb{R}^{3 \times 3}$ we have the linearization 
\begin{equation*}
\begin{split}
   & (\by+\bz) \times \big((\by +\bz)\times A(\by+\bz)\big)  = \by \times (\by \times A\by) 
    \\ & + \bz \times (\by \times A \by) + \by \times (\bz \times A \by) + \by \times (\by \times  A\bz) + \mathcal{O}(\|\bz\|^2) \quad \text{($\|\bz\| \to 0$)}.
\end{split} 
\end{equation*}
Hence for approximating $\mathbf{w}^{j+1}_{h}$ in \eqref{nonlineareq} the Newton iteration is  as follows. Set $\mathbf{w}^{j+1,0}_{h}=\mathbf{u}^{j}_{h}$ and for $l=0,1,2,\ldots$, $\mathbf{w}^{j+1,l+1}_{h}=\mathbf{w}^{j+1,l}_{h} + \bz_h$, with $\bz_h \in
\mathbf{V}^1_{h,0}$ such that for all $\mathbf{v}_{h}\in
\mathbf{V}^1_{h,0}$: 
\begin{equation}
	\begin{split}
 & \frac{2}{\tau}\left(\bz_h , \mathbf{v}_{h} \right)_{h} + \left(\bz_h\times (\mathbf{w}_{h}^{j+1,l} \times  \Tilde{\Laplace}_{h} \mathbf{w}_{h}^{j+1,l}), \mathbf{v}_{h}\right)_{h}
 \\ & + \left(\mathbf{w}_{h}^{j+1,l}\times (\bz_h \times  \Tilde{\Laplace}_{h} \mathbf{w}_{h}^{j+1,l}), \mathbf{v}_{h}\right)_{h}+\left(\mathbf{w}_{h}^{j+1,l}\times (\mathbf{w}_{h}^{j+1,l} \times  \Tilde{\Laplace}_{h} \bz_h), \mathbf{v}_{h}\right)_{h}
 \\ & = \frac{2}{\tau}\left(\mathbf{u}^j_{h},\mathbf{v}_{h}\right)_{h}- \frac{2}{\tau}\left(\mathbf{w}^{j+1,l}_{h} , \mathbf{v}_{h} \right)_{h} -\left(\mathbf{w}_{h}^{j+1,l}\times (\mathbf{w}_{h}^{j+1,l} \times  \Tilde{\Laplace}_{h} \mathbf{w}_{h}^{j+1,l}), \mathbf{v}_{h}\right)_{h}.
\end{split} \label{eq:CPFEM_Newton}
\end{equation}
For the constraint preserving discretization \eqref{nonlineareq} combined with the Newton 
linearization \eqref{eq:CPFEM_Newton} we introduce the notation CPFEM$+$Newton.
In each iteration of the Newton method one has to solve a linear equation  of the form
\begin{equation} \label{Newtonlinear}
	\left(\frac{2}{\tau} M+ C + B + M^{-1}K \right) \underline{z} = f,
\end{equation}
with a symmetric positive definite mass matrix $M$, a symmetric positive definite stiffness matrix $K$ and nonsymmetric matrices  $C$ and $B$. The latter two matrices have conditioning properties comparable to that of a mass matrix. Note that the linear system that has to be solved in Algorithm \ref{algo:HMHF_Bartels} is of the form
\begin{equation} \label{FPlinear}
 \left(\frac{2}{\tau} M+ C\right) \underline{z} = f,
\end{equation}
hence, it does not contain the stiffness matrix component that is present in \eqref{Newtonlinear}.

\paragraph*{Numerical results}
A reference solution $\mathbf{u}_{h}^{\text{ref}}$ is determined in the  same way as in Section~\ref{sec:HMHF_PPFEM}. The discrete approximations  $\mathbf{u}_{h}^j$ are  computed using CPFEM$+$FP. \\
Results for this method are presented in Table \ref{table:HMHF_2D_CrossFEM_FixedPoint_L2_H1_Convergence_h_p1}. These show an optimal convergence rate of order 2 in the $L^2$ norm with respect to $h$, while the convergence rate in the $H^1$ norm is slightly higher than the optimal one. The results are   similar to those of PPFEM, cf.~Table \ref{table:HMHF_2D_PPFEM_L2_H1_Convergence_h_p1} and of TFEM, cf. Table~\ref{table:HMHF_2D_TFEM_L2_H1_Convergence_h_p1}. We use a tolerance  parameter $\epsilon = 10^{-10} $; in this example the fixed point Algorithm~\ref{algo:HMHF_Bartels} then needs on average 3 iterations per time step to satisfy the tolerance criterion. Numerical experiments show that the CFL condition $\tau = \mathcal{O}(h^2)$ that is needed in the analysis is essential. For example, Algorithm~\ref{algo:HMHF_Bartels} does not converge for our example if we use a scaling $\tau \sim h$.  \\
Because of this CFL time step restriction $\tau = \mathcal{O}(h^2)$, it is not possible to study the convergence rate depending on $\tau$ for a fixed $h$ that is small enough such that the spatial error does not dominate.

The convergence rates for the constrained preserving method with the Newton linearization, CPFEM$+$Newton, with a fixed time step $\tau = 10^{-6}$, tolerance parameter $\epsilon=10^{-10}$ and varying $h$, are the same as the results displayed in Table \ref{table:HMHF_2D_CrossFEM_FixedPoint_L2_H1_Convergence_h_p1}. The Newton iteration requires on average two iterations to satisfy the tolerance criterion. We observe that, in this example,  the CFL condition $\tau = \mathcal{O}(h^2)$ is \emph{not} necessary for the convergence of the Newton iteration in each time step. Hence, we can study its convergence behavior for a fixed small mesh size $h$ and varying $\tau$, see Table \ref{table:HMHF_2D_CrossNewton_L2_H1_Convergence_tau_bdf1}. 

The method reaches a $L^2$-error of order $10^{-4}$ and an $H^1$-error of order $10^{-3}$ with a significantly larger time step than PPFEM and TFEM, cf. Tables \ref{table:HMHF_2D_PPFEM_L2_H1_Convergence_tau_bdf1} and \ref{table:HMHF2D_TFEM_L2_H1_Convergence_tau_bdf1}. This is related to the relatively very fast convergence in the beginning of the $\tau$ refinement process (large EOC).  In the last time step we observe that the errors stagnate due to the spatial error dominating.

\begin{table}[ht!]
\centering
\begin{tabular}{ |p{2.5cm}|| p{2.65cm}| p{1cm}|| p{2.65cm}|p{1cm}|}
 \hline
  $\tau = 10^{-6}$ & $\lVert \mathbf{u}_{h}^{J} - \mathbf{u}^{\text{ref},J}_{h} \rVert_{L^2}$ & EOC & $\lVert \mathbf{u}_{h}^{J} - \mathbf{u}^{\text{ref},J}_{h}  \rVert_{H^1}$ & EOC\\
 \hline
  $h=2^{-2}$  & $6.3361e-02$ & $-$  & $3.0327e-01$ & $-$ \\
 $h=2^{-3}$  & $1.5403e-02$ & $2.04$  & $1.0703e-01$ & $1.50$ \\
 $h=2^{-4}$ & $3.6554e-03$ & $2.08$ & $3.9342e-02$ & $1.44$\\
 $h=2^{-5}$ & $9.2286e-04$ & $1.99$ & $1.3897e-02$ & $1.50$\\
$h=2^{-6}$ & $2.3430e-04$ & $1.98$  & $5.0386e-03$ & $1.46$ \\
 \hline 
\end{tabular}
\caption{Error for CPFEM$+$FP, \eqref{nonlineareq} combined with Algorithm \ref{algo:HMHF_Bartels}, with $\epsilon = 10^{-10}$;  $\tau = 10^{-6}$ fixed.}
\label{table:HMHF_2D_CrossFEM_FixedPoint_L2_H1_Convergence_h_p1}
\end{table}

\begin{table}[ht!]
	\centering
	\begin{tabular}{ |p{2.5cm}|| p{2.65cm}| p{1cm}|| p{2.65cm}|p{1cm}|}
		\hline
		$h=2^{-6}$ & $\lVert \mathbf{u}_{h}^{J} - \mathbf{u}^{\text{ref},J}_{h} \rVert_{L^2}$ & EOC & $\lVert \mathbf{u}_{h}^{J} - \mathbf{u}^{\text{ref},J}_{h}  \rVert_{H^1}$ & EOC\\
		\hline
		$\tau=5 \cdot 10^{-2}$  & $9.9508e-03$ & $-$  & $1.2203e-01$ & $-$ \\ 
		$\tau=2.5 \cdot 10^{-2}$  & $1.9746e-03$ & $2.33$  & $4.1280e-02$ & $1.56$ \\
		$\tau=1.25 \cdot 10^{-2}$ & $4.0463e-04$ & $2.29$ & $1.4794e-02$ & $1.48$\\
		$\tau=6.25 \cdot 10^{-3}$ & $2.0388e-04$ & $0.99$ & $6.4949e-03$ & $1.19$\\
		$\tau=3.125 \cdot 10^{-3}$ & $2.2208e-04$ & $-0.12$  & $5.0659e-03$& $0.36$ \\
		\hline 
	\end{tabular}
	\caption{Error for CPFEM$+$Newton,  \eqref{nonlineareq} combined with 
 \eqref{eq:CPFEM_Newton}, with $\epsilon = 10^{-10}$; $h=2^{-6}$ fixed. }
	\label{table:HMHF_2D_CrossNewton_L2_H1_Convergence_tau_bdf1}
\end{table}

\section{Outlook: Numerical experiment for a  blow up case} \label{SecOutlook}
In this section, as a motivation for further research,  we present results of a few numerical experiments in which the methods PPFEM (pointwise projection), TFEM (tangent space approach), CPFEM$+$FP (constraint preserving discretization combined with fixed point) and CPFEM$+$Newton (constraint preserving discretization combined with Newton) are applied to a HMHF problem in which blow up occurs. \\ 
From the results of \cite{ChangDingYe1992} we know that the derivative of the solution of the harmonic map heat flow problem \eqref{HMHFeq} blows up in finite time at the origin if the initial condition is of the form \eqref{eq:HMHF_InitialCondSphericallySymmetric} and satisfies $|u_0(1)|>\pi$. We investigate how the above-mentioned discretization methods for the two-dimensional problem \eqref{HMHFeq} behave  when approaching the blow up time.

We choose the initial condition
\begin{align}
	u_0(r) = 4.25 \pi r^2. \label{eq:InitCond_blowup}
\end{align}

First, we discretize the  reduced harmonic map heat flow equation \eqref{eq:HMHF_1DPDE} using the method \eqref{eq:HMHF_1D_full_discr} for $p=1$ and $k=1$. Since we know that the singularity can only occur at the origin, the spatial grid on $[0,1]$ is strongly refined close to the origin. Numerical experiments in \cite[Section 4.1]{HaynesHuangZegeling2013} yield an approximate blow up time $0.0835$ for the solution of \eqref{eq:HMHF_1DPDE} with the initial condition \eqref{eq:InitCond_blowup}. Hence we choose the end time $T=0.083$.
\begin{figure}[ht!]
	\centering
	\includegraphics[width=\textwidth]{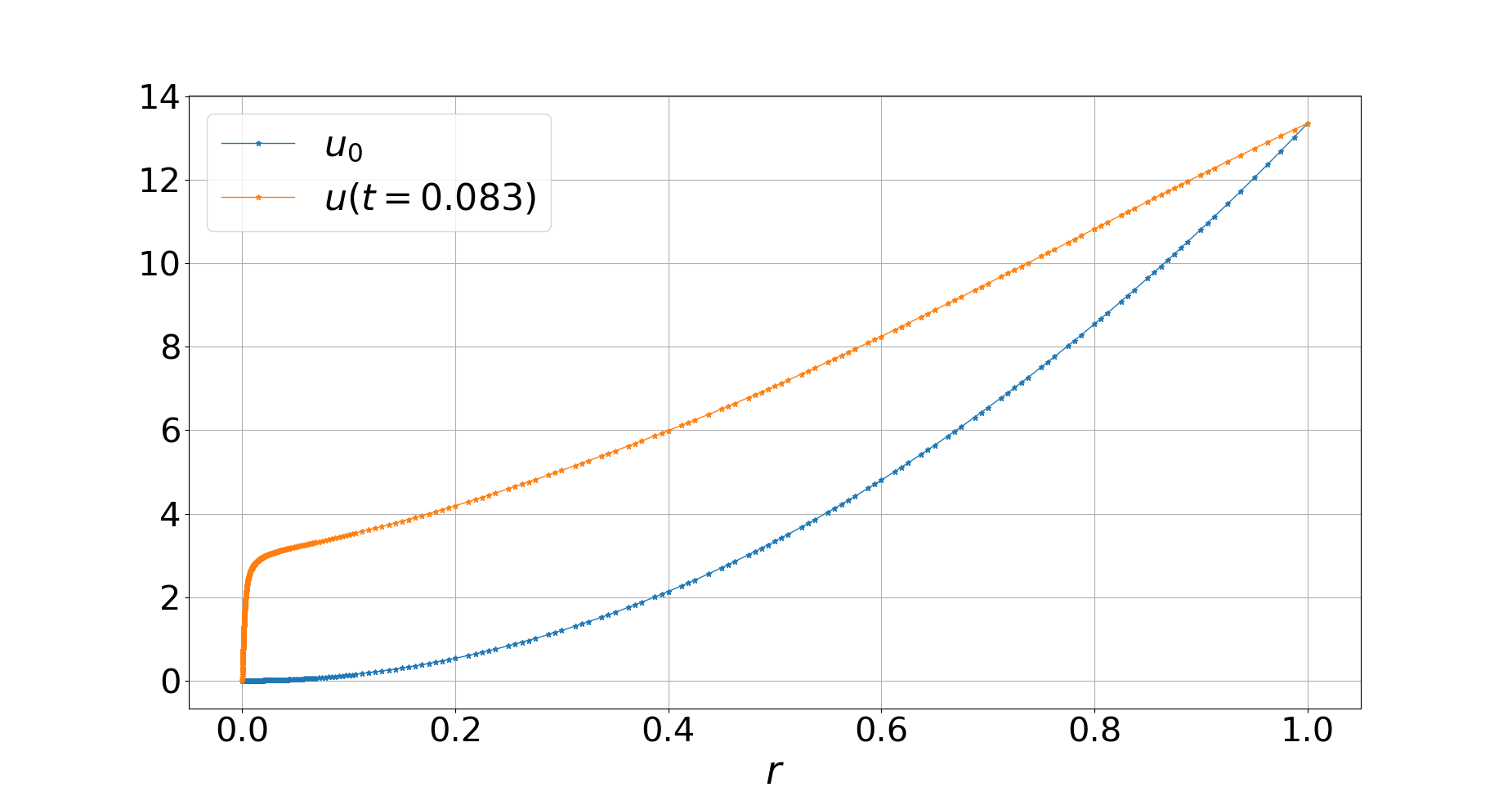}
	\caption{Blow up: discrete solution of \eqref{eq:HMHF_1DPDE} with initial condition \eqref{eq:InitCond_blowup} }
	\label{fig:RSHMHF_blowup}
\end{figure}
Figure \ref{fig:RSHMHF_blowup} shows that the numerical scheme \eqref{eq:HMHF_1D_full_discr} with a strongly refined grid close to the origin can capture the singularity formation well. We transform this discrete solution of the reduced problem to the two-dimensional vector valued solution using the transformation \eqref{eq:RSHMHF_SphericalTransformation}. Here, the two-dimensional mesh is also strongly refined around the origin. The result is shown in Figure \ref{fig:HMHF_blowup_transformation} and will serve as the reference solution for this section.
\begin{figure}[ht!]
	\centering
	\begin{subfigure}{0.49\linewidth}
		\centering
		\includegraphics[
		width=\linewidth
		]{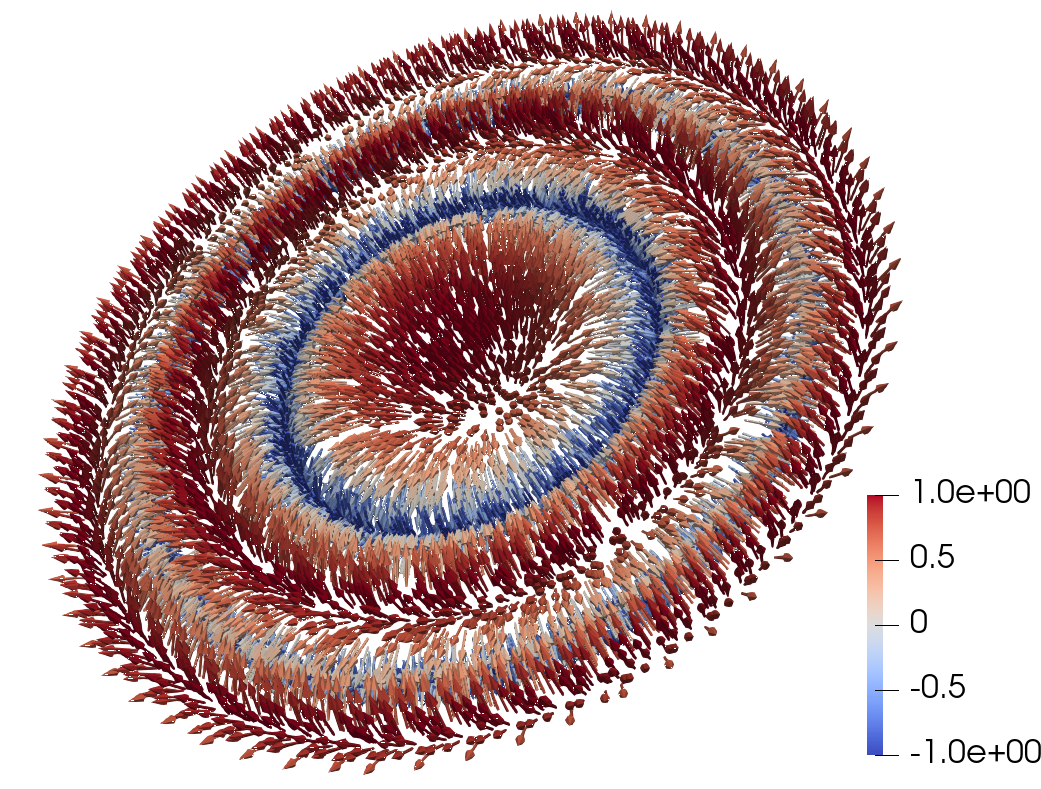}
		\caption{Initial condition}
		\label{fig:2d_blowup_init}
	\end{subfigure}
	\hfill
	\begin{subfigure}{0.49\linewidth}
		\centering
		\includegraphics[
		width=\linewidth
		]{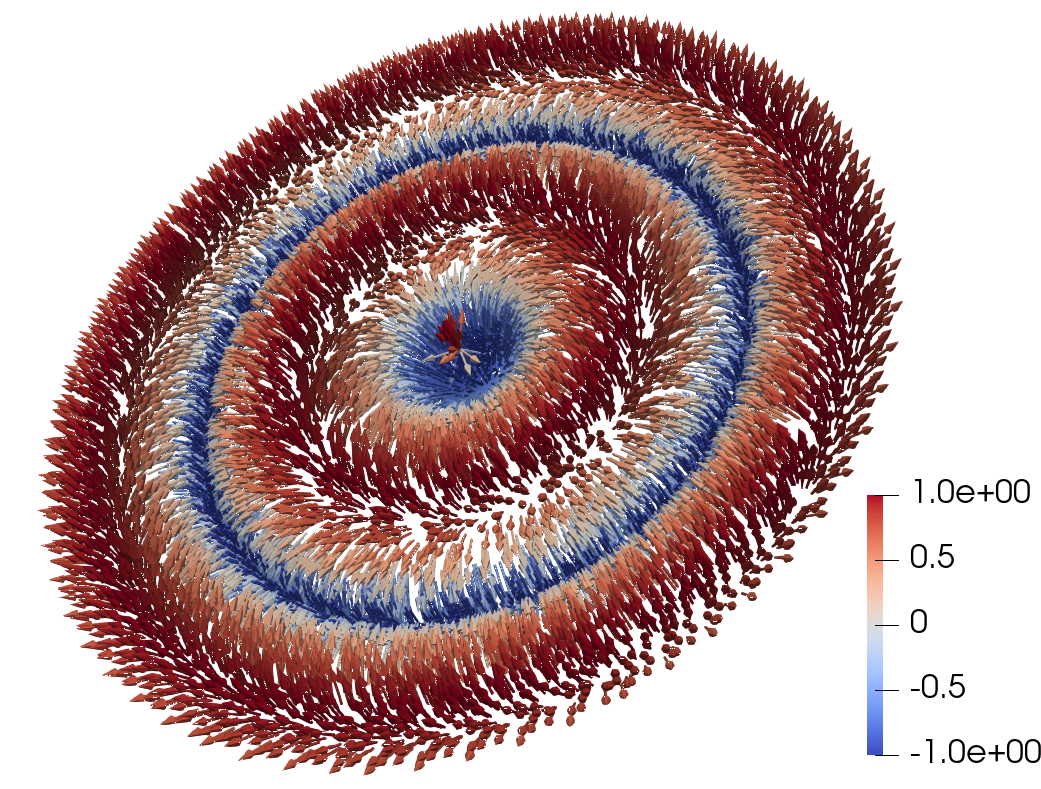}
		\caption{$t=0.083$}
		\label{fig:2d_blowup}
	\end{subfigure}
	\begin{subfigure}{0.49\linewidth}
		\centering
		\includegraphics[
		width=\linewidth
		]{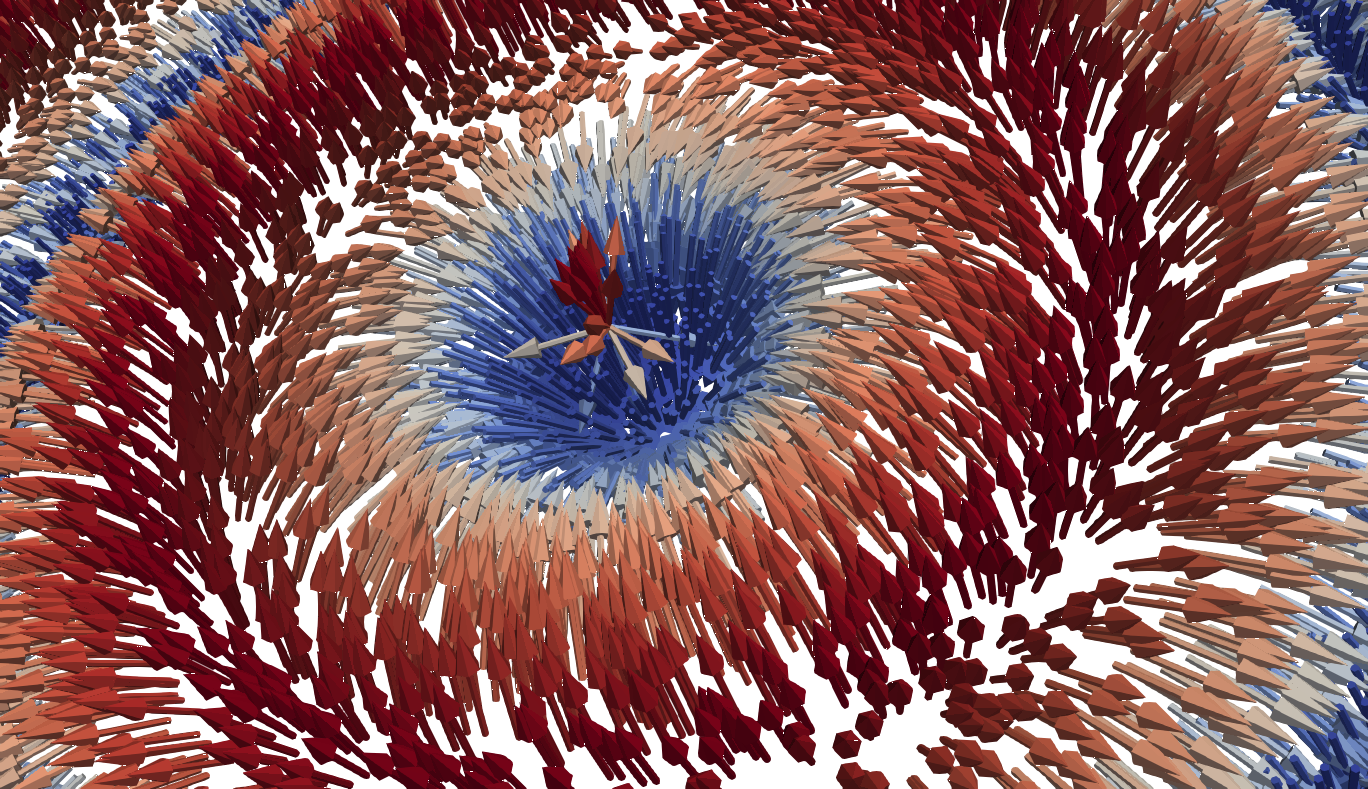}
		\caption{Zoom into the neighborhood of the origin from \ref{fig:2d_blowup}}
	\end{subfigure}
	\caption{Transformation of discrete solution from Figure \ref{fig:RSHMHF_blowup} to the two-dimensional vector-valued solution using \eqref{eq:RSHMHF_SphericalTransformation}; color represents the value of the third component}
	\label{fig:HMHF_blowup_transformation}
\end{figure}

Now we run the methods for the two-dimensional problem \eqref{HMHFeq} of Section \ref{sec:FEM_HMHF} using the initial condition \eqref{eq:HMHF_InitialCondSphericallySymmetric} corresponding to \eqref{eq:InitCond_blowup}. We choose BDF1 and linear finite elements ($p=1$) with a fixed time step $\tau = 10^{-4}$ and mesh size $h=10^{-2}$  in all methods. For the CPFEM$+$FP and CPFEM$+$Newton we choose $\epsilon = 10^{-10}$ for the stopping criterion. 
For CPFEM$+$FP we take a much smaller time step, $\tau = 10^{-5}$, to guarantee convergence of the fixed point  iteration.  
 In CPFEM$+$FP, Algorithm \ref{algo:HMHF_Bartels} needed for the first few time steps on average 25 iterations and later around 10 iterations to satisfy the stopping criterion in each time step. The Newton iteration  required on average 4 iterations in each time step. 

The results obtained by the different  methods are shown in Figures \ref{fig:HMHF_TFEM_blowup} - \ref{fig:HMHF_Newton_blowup}. The methods CPFEM$+$FP and  CPFEM$+$Newton produce very similar results, hence we show results only for the latter. We observe that \emph{the discrete solutions of all methods do \emph{not} match the reference solution} in Figure \ref{fig:2d_blowup} close to the approximated blow up time. Until $t \approx 0.02$  all discrete solutions behave in the same way and show qualitative good agreement with the reference solution, cf. subfigures (a) in Figures \ref{fig:HMHF_TFEM_blowup} - \ref{fig:HMHF_Newton_blowup}.
In TFEM  the symmetry of the reference solution \eqref{eq:HMHF_SphericallySymmtricSolution} is broken after $t\approx 0.02$, while this symmetry breaking occurs later for the PPFEM and CPFEM$+$Newton methods, namely after $t\approx 0.03$. The latter two methods yield qualitatively similar results for all $t \in [0,T]$. In particular, both display a swirling motion: The vector fields are winding around the origin. Both methods have in common that the unit length constraint is (up to rounding error) exact at the discretization points. In contrast, the  TFEM method shows a different solution behavior, cf. Figure~\ref{fig:HMHF_TFEM_blowup}. Note that in TFEM the unit length constraint is treated in  a more relaxed implicit way. In Figure ~\ref{fig:blowup_HMHF_unit_length} we show the error in the unit length constraint at the  end $t=0.083$, i.e., $\big|1-|\mathbf{u}_h^J|\big|$ for the methods PPFEM and TFEM.  The unit length violation is strongest  around the origin where the singularity occurs. For TFEM the unit length violation is overall stronger compared to the PPFEM.

Finally, we consider the Dirichlet energy $\frac{1}{2}\int_\Omega |\nabla \mathbf{u}|^2\, dx$. 
The energy of the reference solution and of the discrete solutions  are shown in Figure \ref{fig:blowup_energie}. It turns out that the symmetry breaking corresponds to a significant  drop of the Dirichlet energy compared to the energy of the reference solution. Note that the energy behavior of the PPFEM and CPFEM$+$Newton discrete solutions is very  similar, whereas the TFEM discrete solution behaves differently.

For all of these methods decreasing the time step and/or refining the mesh size shows essentially the same results.  In the cases of TFEM and PPFEM using their BDF2 and $p=2$ (quadratic finite elements) variants also yield similar results.

These results indicate that approximating the solution of a HMHF problem with finite time blow up is a challenging task. Already in this relatively simple example it turns that direct application of the three established methods that we compared above does not yield satisfactory results. The different numerical approximations obtained by the different methods suggest that the continuous problem  \eqref{HMHFeq} is  poorly conditioned and has besides the  rotationally symmetric  solution \eqref{eq:HMHF_SphericallySymmtricSolution} other nearby ``solution branches''. Note that already in the discretization of the initial condition \eqref{eq:HMHF_InitialCondSphericallySymmetric} the exact rotational symmetry may be lost in the three methods for the vector valued problem \eqref{HMHFeq}. It might well be that if these methods are combined with (strong) local refinements in space (close to the origin) and time (close to the blow up point in time) they are able to preserve the symmetry property and yield accurate approximations of the continuous solutions. We consider all these issues to be interesting topics for further research. 

\begin{figure}[ht!]
	\centering
	
	\begin{subfigure}{0.49\linewidth}
		\centering
		\includegraphics[
		width=\linewidth
		]{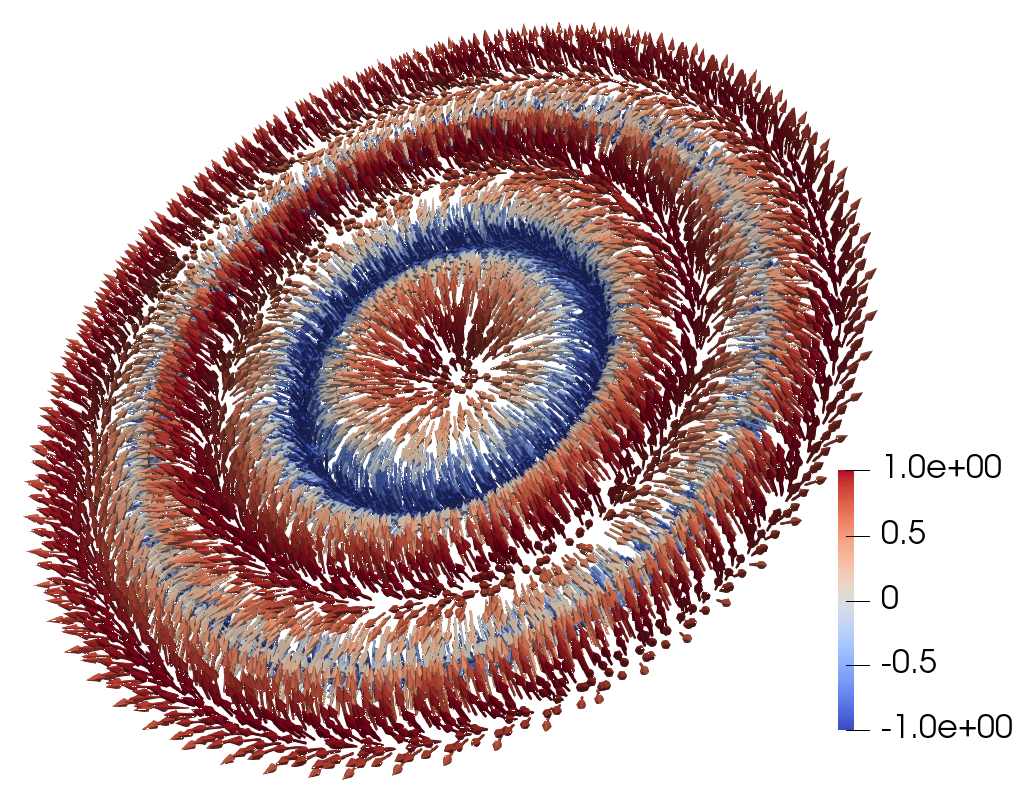}
		\caption{$t=0.02$}
	\end{subfigure}
	\hfill
	\begin{subfigure}{0.49\linewidth}
		\centering
		\includegraphics[
		width=\linewidth
		]{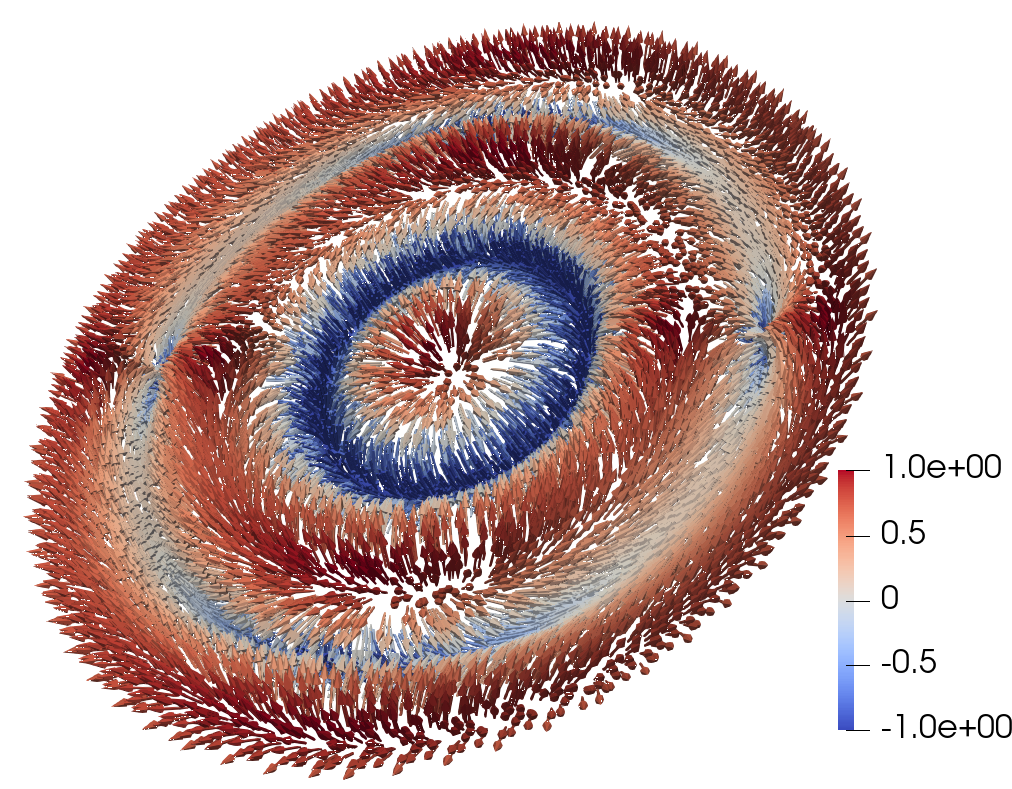}
		\caption{$t=0.03$}
	\end{subfigure}
	
	\begin{subfigure}{0.49\linewidth}
		\centering
		\includegraphics[
		width=\linewidth
		]{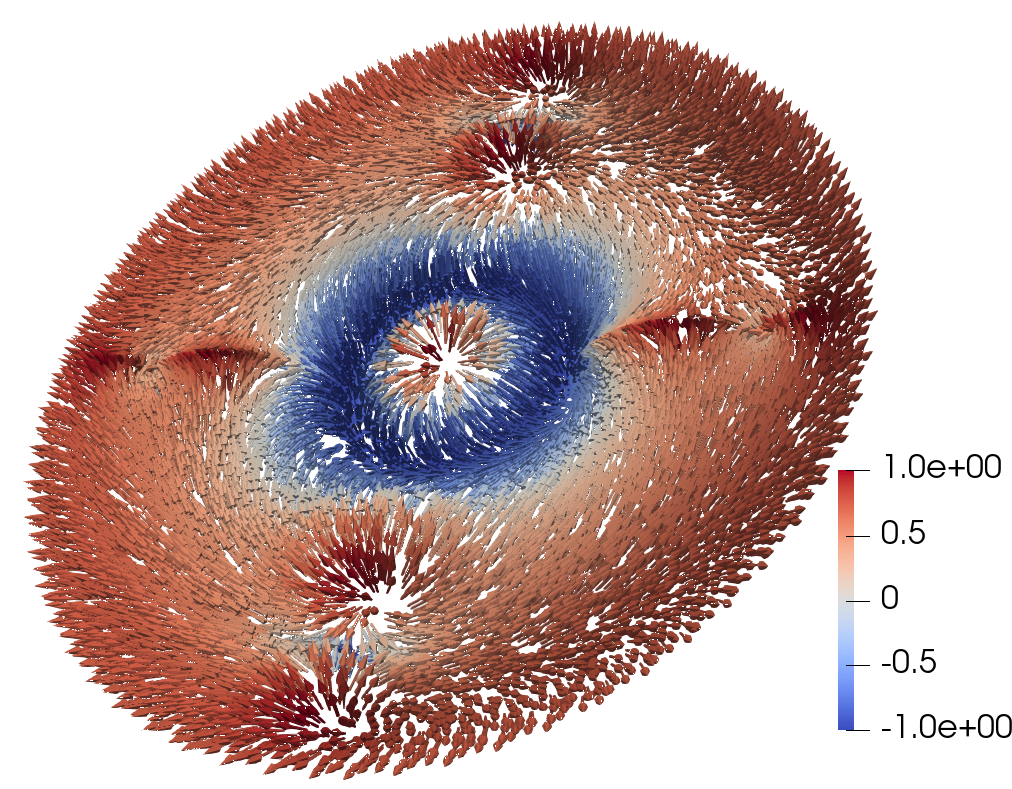}
		\caption{$t=0.05$}
	\end{subfigure}
	\hfill
	\begin{subfigure}{0.49\linewidth}
		\centering
		\includegraphics[
		width=\linewidth
		]{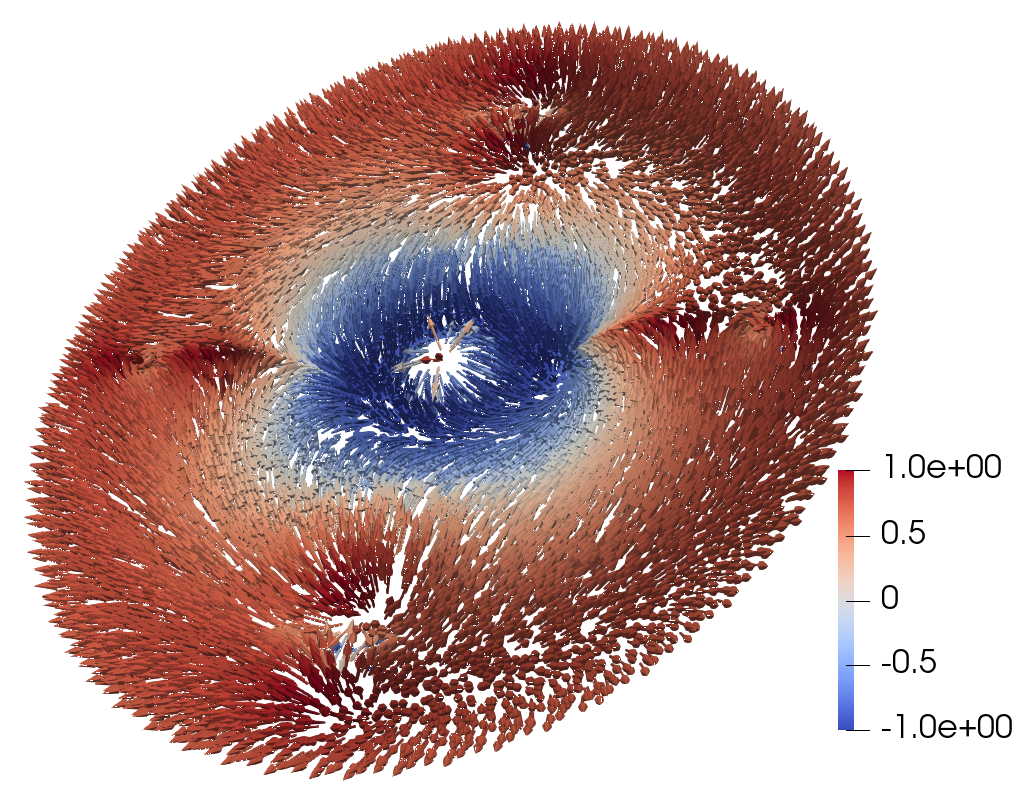}
		\caption{$t=0.083$}
	\end{subfigure}
	
	\caption{Blow up case: TFEM (\eqref{eq:HMHF_TFEM_1}-\eqref{eq:HMHF_TFEM_2} and \eqref{eq:HMHF_TFEM_update}) with initial condition corresponding to \eqref{eq:InitCond_blowup}; color represents the value of the third component }
	\label{fig:HMHF_TFEM_blowup}
\end{figure}

\begin{figure}[ht!]
	\centering
	
	\begin{subfigure}{0.49\linewidth}
		\centering
		\includegraphics[
		width=\linewidth
		]{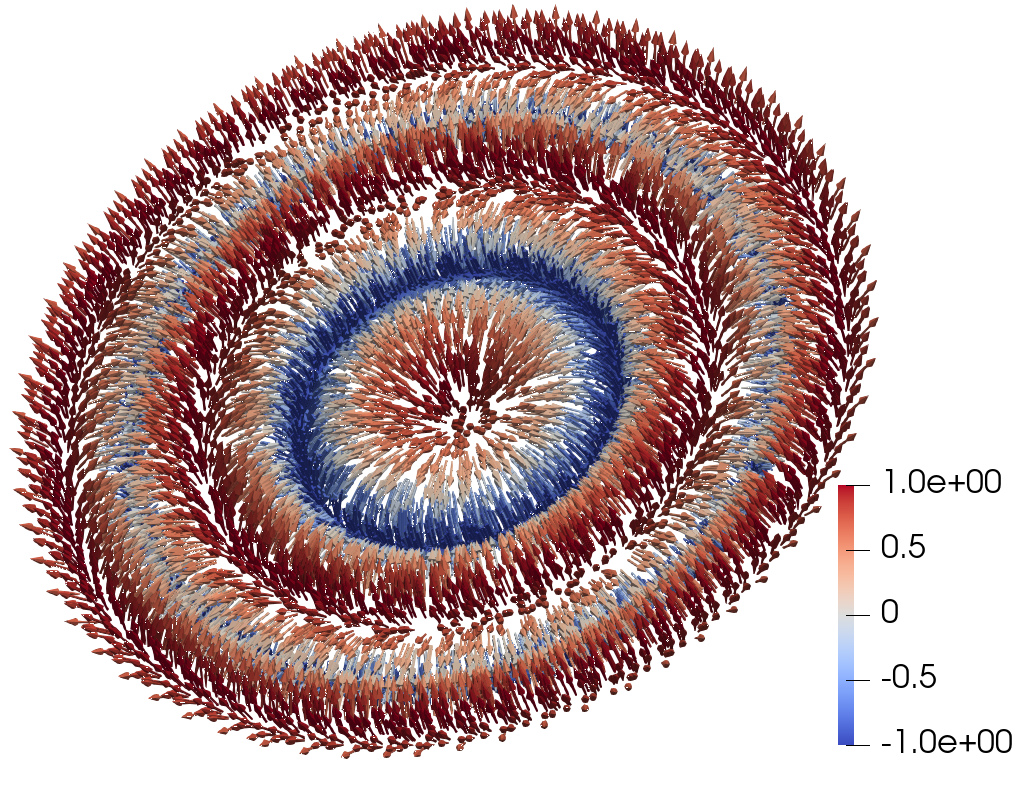}
		\caption{$t=0.02$}
	\end{subfigure}
	\hfill
	\begin{subfigure}{0.49\linewidth}
		\centering
		\includegraphics[
		width=\linewidth
		]{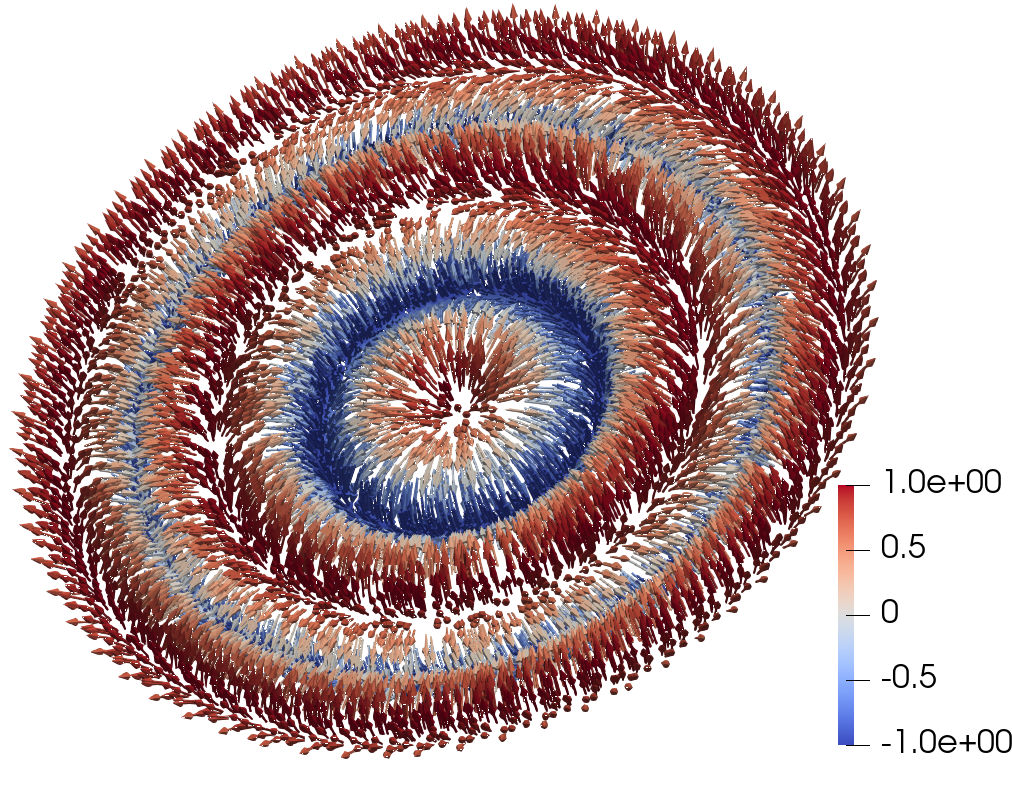}
		\caption{$t=0.03$}
	\end{subfigure}
	
	\begin{subfigure}{0.49\linewidth}
		\centering
		\includegraphics[
		width=\linewidth
		]{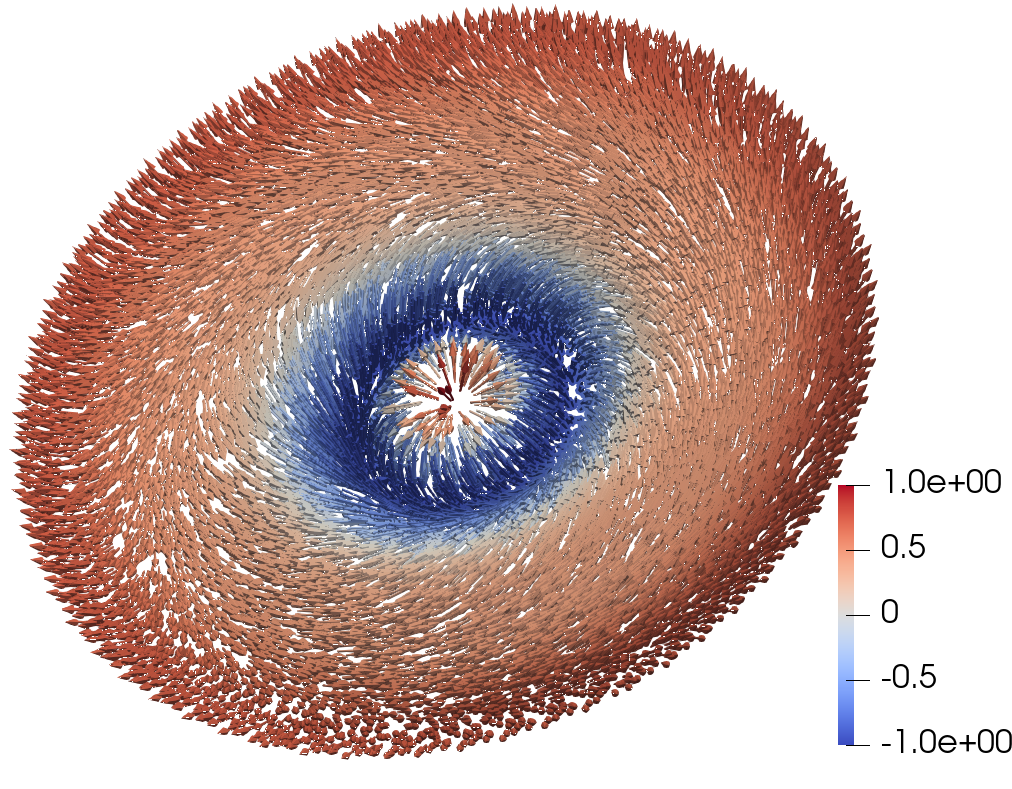}
		\caption{$t=0.05$}
	\end{subfigure}
	\hfill
	\begin{subfigure}{0.49\linewidth}
		\centering
		\includegraphics[
		width=\linewidth
		]{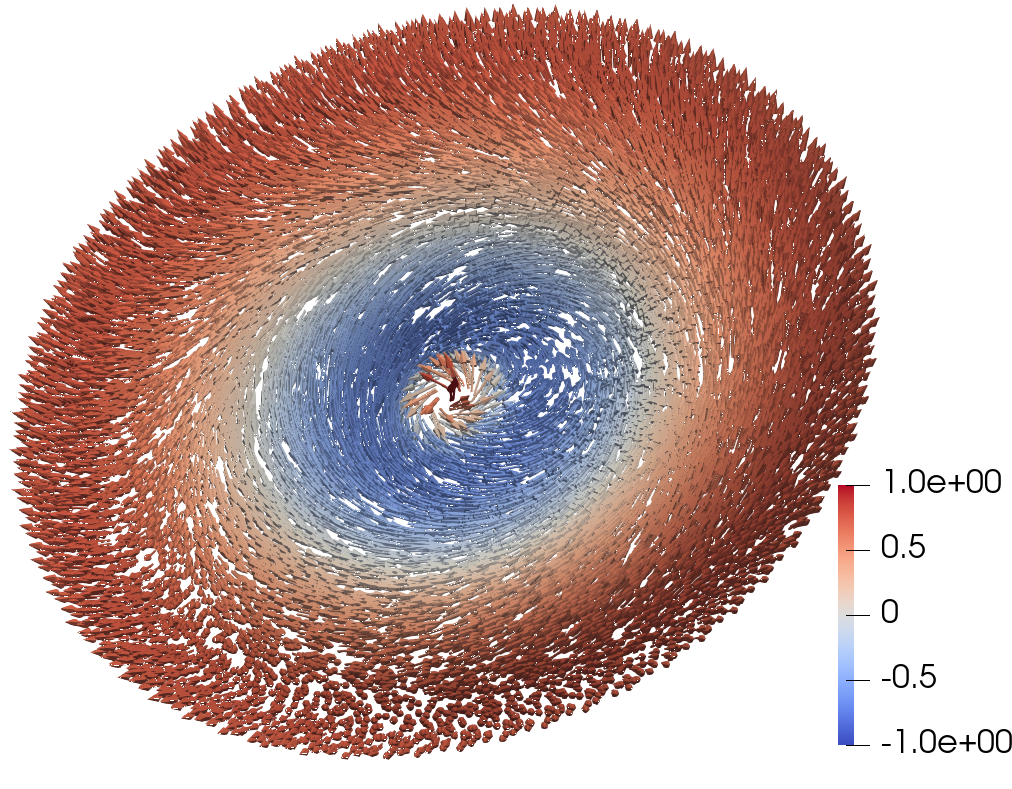}
		\caption{$t=0.083$}
	\end{subfigure}
	
	\caption{Blow up case: PPFEM \eqref{eq:HMHF_PPFE}-\eqref{eq:HMHF_PPFE_normalization} with initial condition corresponding to  \eqref{eq:InitCond_blowup}; color represents the value of the third component }
	\label{fig:HMHF_PP_blowup}
\end{figure}

\begin{figure}[ht!]
	\centering
	
	\begin{subfigure}{0.49\linewidth}
		\centering
		\includegraphics[
		width=\linewidth
		]{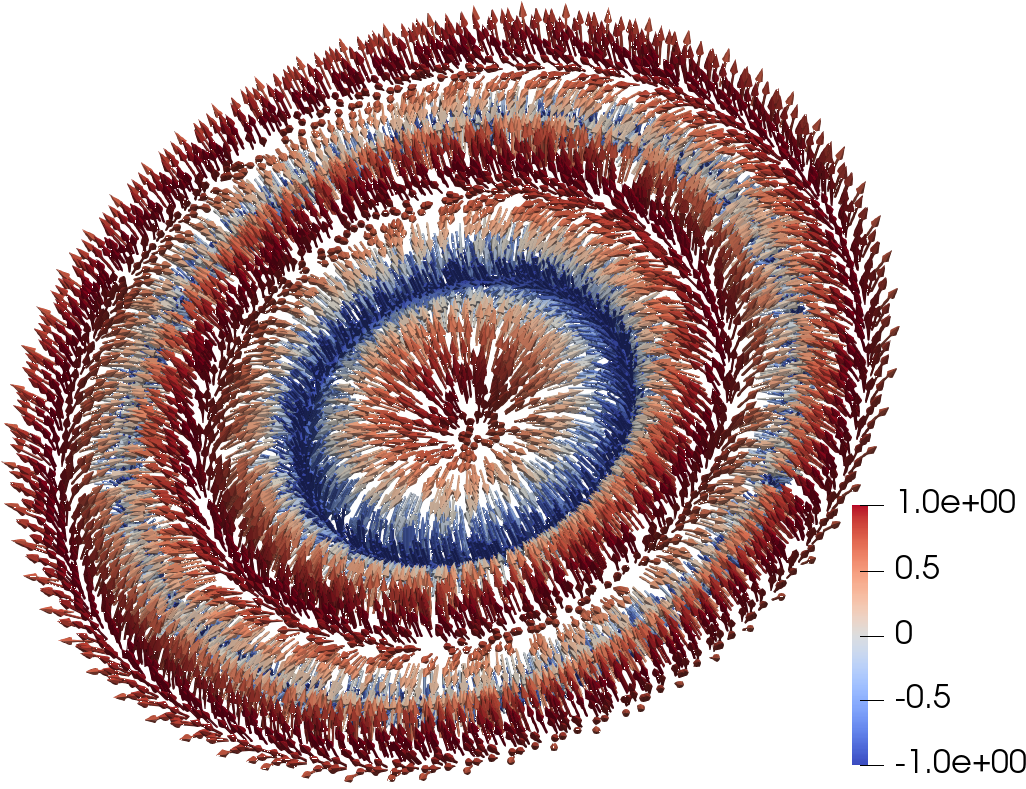}
		\caption{$t=0.02$}
	\end{subfigure}
	\hfill
	\begin{subfigure}{0.49\linewidth}
		\centering
		\includegraphics[
		width=\linewidth
		]{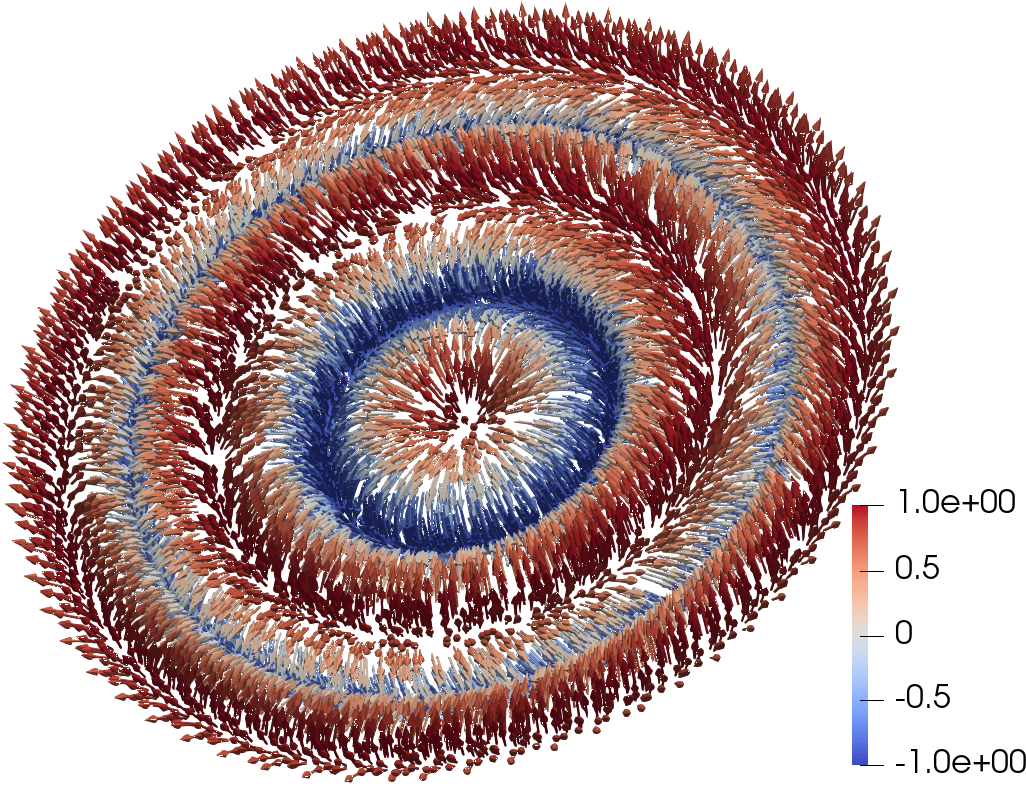}
		\caption{$t=0.03$}
	\end{subfigure}
	
	\begin{subfigure}{0.49\linewidth}
		\centering
		\includegraphics[
		width=\linewidth
		]{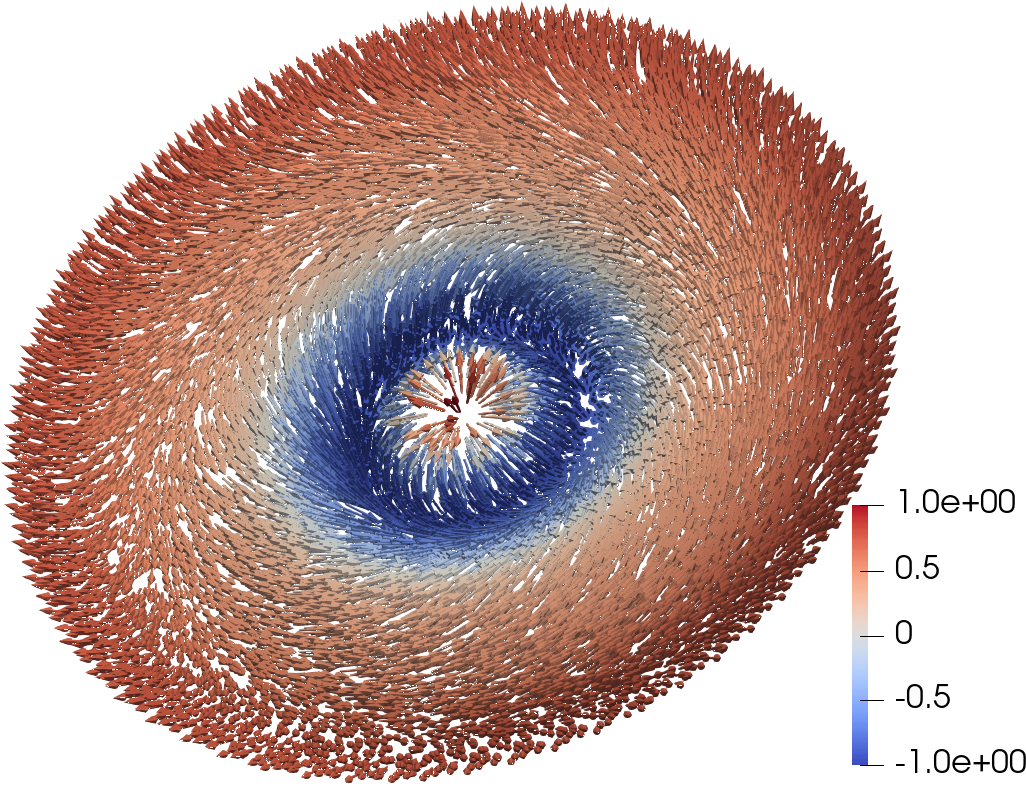}
		\caption{$t=0.05$}
	\end{subfigure}
	\hfill
	\begin{subfigure}{0.49\linewidth}
		\centering
		\includegraphics[
		width=\linewidth
		]{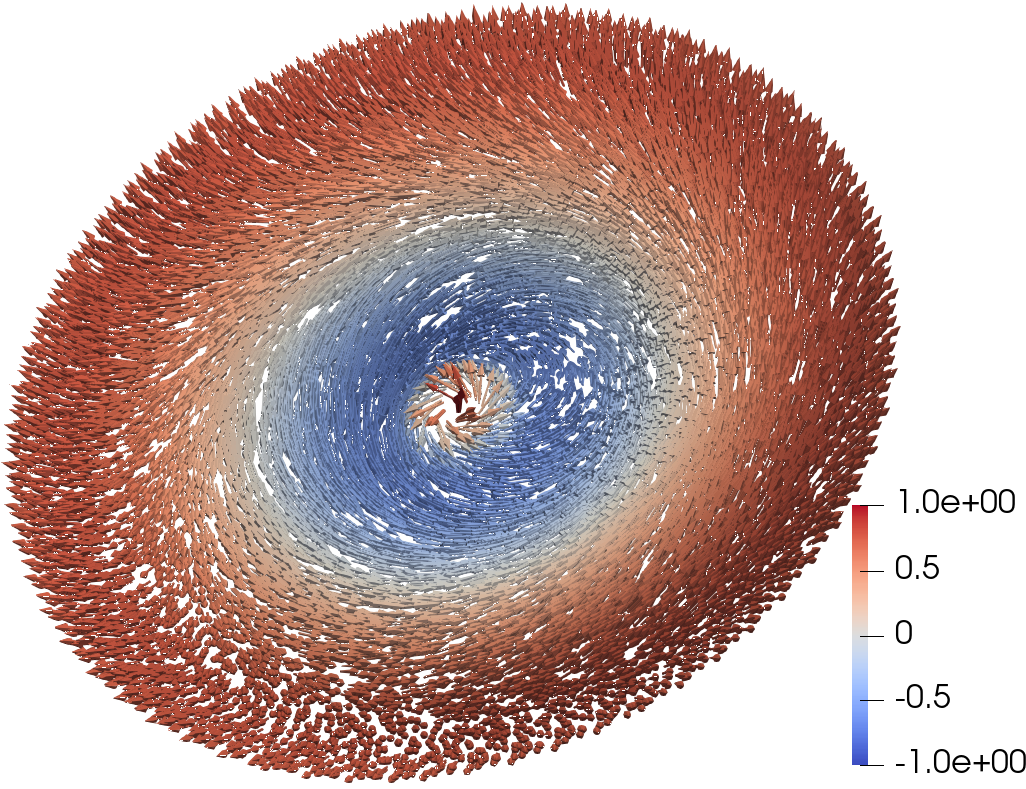}
		\caption{$t=0.083$}
	\end{subfigure}
	
	\caption{Blow up case: CPFEM$+$Newton,  \eqref{nonlineareq} combined with 
 \eqref{eq:CPFEM_Newton}, with initial condition corresponding to \eqref{eq:InitCond_blowup}; color represents the value of the third component }
	\label{fig:HMHF_Newton_blowup}
\end{figure}
\begin{figure}[ht!]
	\centering
	
	\begin{subfigure}{0.49\linewidth}
		\centering
		\includegraphics[
		width=\linewidth
		]{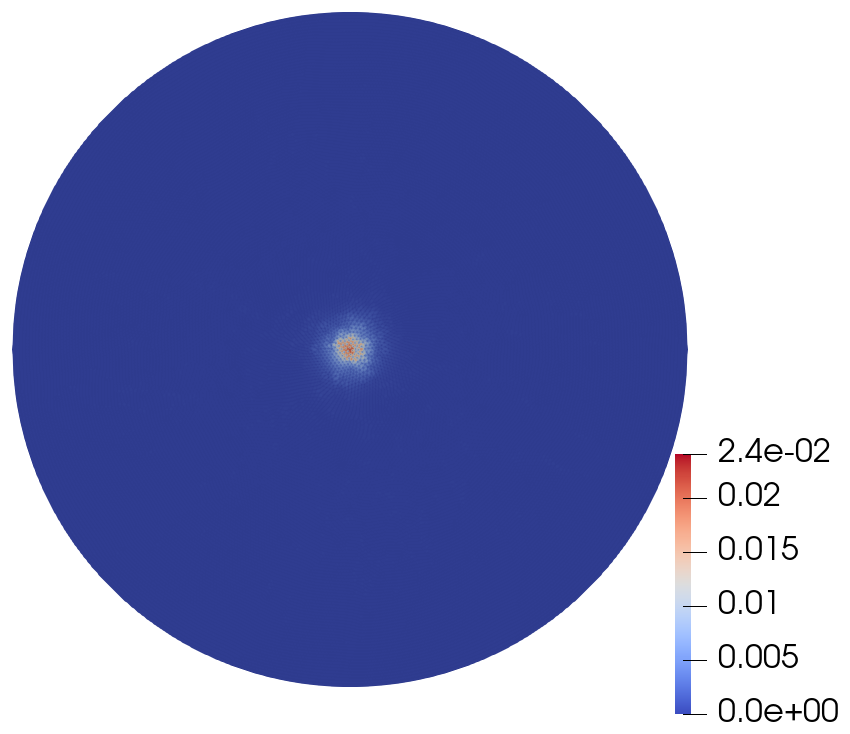}
		\caption{PPFEM \eqref{eq:HMHF_PPFE}-\eqref{eq:HMHF_PPFE_normalization} at $t=0.083$}
	\end{subfigure}
	\begin{subfigure}{0.49\linewidth}
		\centering
		\includegraphics[
		width=\linewidth
		]{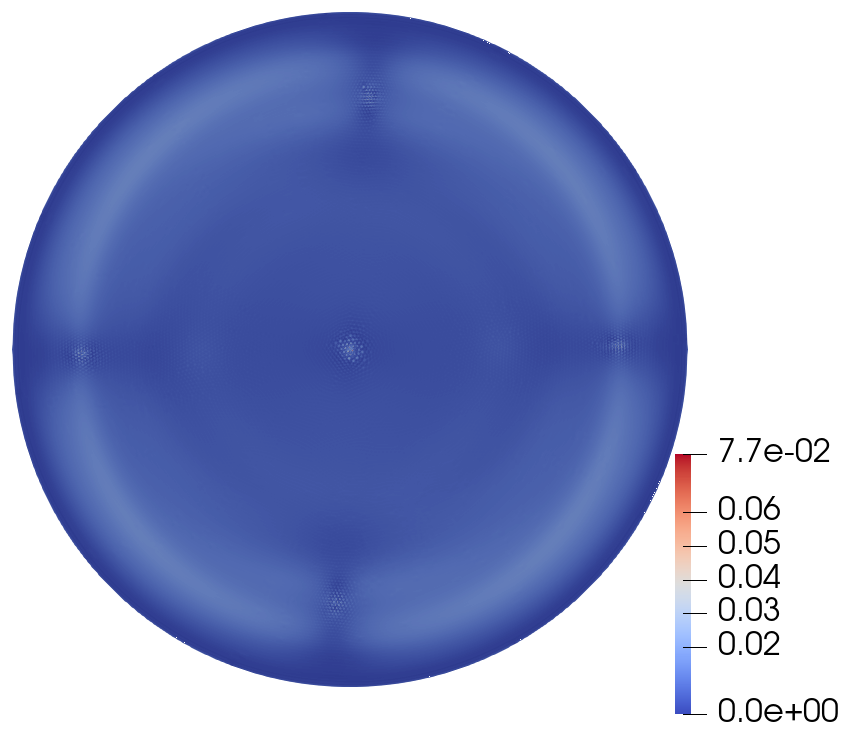}
		\caption{TFEM (\eqref{eq:HMHF_TFEM_1}-\eqref{eq:HMHF_TFEM_2} and \eqref{eq:HMHF_TFEM_update}) at $t=0.083$}
	\end{subfigure}

	\caption{Blow up case: unit length violation $|1-|\mathbf{u}_h^J| \,|$}
	\label{fig:blowup_HMHF_unit_length}
\end{figure}
\begin{figure}[ht!]
	\centering
	\includegraphics[width=\textwidth]{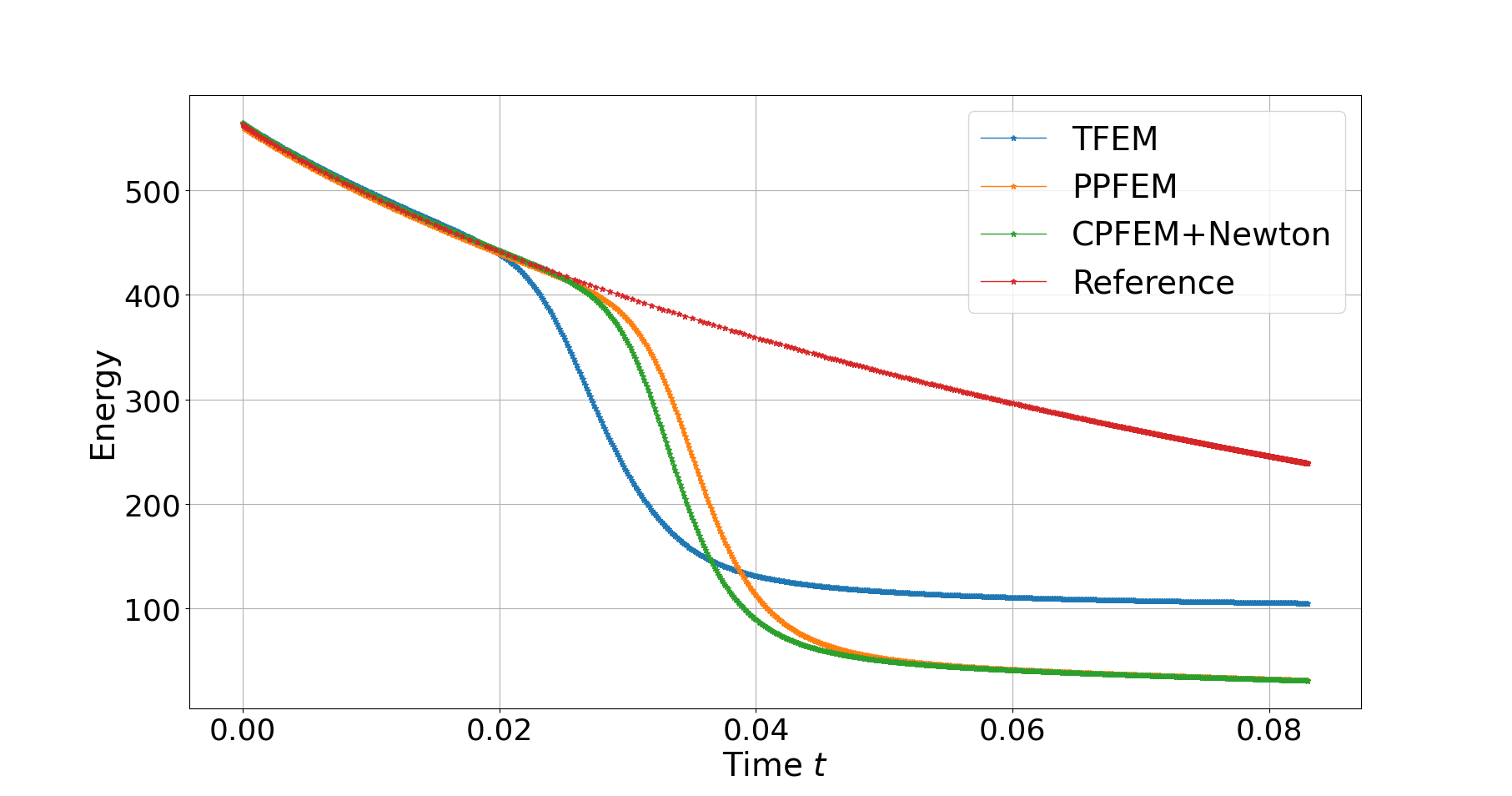}
	\caption{Blow up case: Energies of the discrete solutions from Figures \ref{fig:HMHF_blowup_transformation}-\ref{fig:HMHF_Newton_blowup}}
	\label{fig:blowup_energie}
\end{figure}

\section{Comparison of methods} \label{SecComparison}

We summarize and compare properties of the three methods PPFEM, TFEM and CPFEM, where for the latter we have the two variants CPFEM$+$FP and CPFEM$+$Newton. We address three aspects: discretization accuracy (smooth regime), computational work per time step and application to a blow up example.\\
\emph{Discretization accuracy}. 
All three methods show similar convergence behavior with BDF1 und piecewise linear finite element spaces. 
The method CPFEM$+$FP suffers from a severe time step restriction needed for convergence of the fixed point iteration. This restriction is not needed for convergence of the  CPFEM$+$Newton method. PPFEM with BDF2 combined with piecewise quadratic finite elements shows a significantly higher rate of convergence  in the mesh size than expected. The versions of TFEM that use BDF2 and piecewise quadratic finite elements show the optimal rates of convergence, both with respect to time step size and mesh size. Higher order BDF schemes do not satisfy the relation \eqref{eq:HMHF_2D_CrossFEM_UnitLength} and it is not clear how to extend the CPFEM to higher order BDF methods. For the methods  PPFEM and TFEM higher order BDF variants are available. 
\\
\emph{Computational work per iteration}. In the methods PPFEM and TFEM one has to solve only one linear system per time step, whereas in CPFEM, due to the inner iteration,  one typically has   more than one linear system to solve. In the inner iteration of  CPFEM$+$Newton the linear system for the $\mathbf{u}$-unknown is of the form \eqref{Newtonlinear}, whereas in CPFEM$+$FP the linear system is of the form \eqref{FPlinear}. Since in the latter the stiffness matrix part is not present, the conditioning of the system matrix is better than in \eqref{Newtonlinear}, which often is an advantage if one uses an iterative solver.  
Despite this better conditioning property, the method  CPFEM$+$FP is  computationally the most expensive  one since one has to satisfy the severe time step restriction and in each time step one needs to solve several ($2$ to $4$) of these linear systems.  In our comparative study this method turns out to be the least efficient one. 
In PPFEM one has to solve one linear system of a similar form as in \eqref{Newtonlinear} for the $\mathbf{u}$-unknown. The computational work needed for the renormalization is  negligible. Since in   CPFEM$+$Newton one typically needs more than one Newton iteration the computational costs per iteration for this method are higher than for PPFEM. In TFEM one has to solve a saddle point system for the $\mathbf{u}$-unknown and the Lagrange multiplier, cf. \eqref{eq:HMHF_TFEM_1}-\eqref{eq:HMHF_TFEM_2}. The number of additional unknowns due to the Lagrange multiplier is approximately one third of the number of unknowns needed for approximation of $\mathbf{u}$. A nice structural property of the system matrix is that it is always symmetric and the block corresponding to the $\mathbf{u}$ unknowns is symmetric positive definite. Due to its higher dimension, solving this system is in general more expensive than solving the system in PPFEM. In our implementation (where we used a direct solver) the PPFEM has the least computational work per iteration. 
\\
\emph{Application to a blow up example}. We applied the three methods to a relatively simple HMHF problem with finite time blow up, cf. Section~\ref{SecOutlook}. The solution of this problem has a rotational symmetry property.
In all three methods this symmetry property is lost and the resulting numerical solutions are ``wrong'' in the sense that these approximations are qualitatively different from the symmetric reference solution. The symmetry breaking occurs later for the PPFEM and CPFEM methods as for  TFEM. This may be related to the fact that in the former two methods the unit length constraint is (up to rounding error) exact at the discretization points, wheras in TFEM this constraint is treated in  a more relaxed implicit way.
\\[2ex]
{\bf Acknowledgements} The authors acknowledge funding by the Deutsche Forschungsgemeinschaft (DFG, German Research Foundation) – project number 442047500 – through the Collaborative Research Center “Sparsity and Singular
Structures” (SFB 1481).

\begin{appendices}    
\section{Proof of Lemma~\ref{leminfsup}} \label{appendix:InfSupStab} 
Take $w_h \in V_{h,0}^p$. The constants below (all denoted by   $c$) may depend on $\bu$, but are independent of $w_h$, $\tau$, $h$. Let $I_h: C(\bar D;\mathbb{R}^3) \to \bV_{h,0}^p$ be the componentwise nodal interpolation in the finite element space. We write $\bu(t_j)=\bu(t_j,\cdot)$. For $K \in \mathcal{T}_h$ we obtain, using $|w_h|_{H^{p+1}(K)}=0$ and an inverse inequality:
\begin{align*}
  \|I_h(\bu(t_j)w_h)-\bu(t_j)w_h\|_{L^2(K)} & \leq c h^{p+1}|\bu(t_j)w_h|_{H^{p+1}(K)} \\
   & \leq c h^{p+1} \sum_{\ell =0}^p |\bu(t_j)|_{W^{p+1-\ell, \infty}(K)} |w_h|_{H^{\ell}(K)} \\
    & \leq c h \|\bu\|_{W^{p+1,\infty}(K)}\|w_h\|_{L^2(K)}.
\end{align*}
From this we obtain  (with $c=c(\bu)$)
\begin{equation} \label{h1}
 \|I_h(\bu(t_j)w_h)-\bu(t_j)w_h\|_{L^2} \leq c h \|w_h\|_{L^2}.
\end{equation}
We restrict to $\tau \in (0,\varepsilon_1]$, $h \in (0,\varepsilon_2]$ with suitable $ \varepsilon_1>0$, $\varepsilon_2 >0$. Define $\delta_{\tau,h}:= \max_{0 \leq j \leq J} \|\bu_h^j - \bu(t_j)\|_{L^\infty}$.  Due to assumption~\eqref{condconv} the size of $\delta_{\tau,h}$ is controlled by $\varepsilon_1$ and $\varepsilon_2$.  Using $|\bu(t_n)|=1$ on $\Omega$ for all $0 \leq n \leq J$, we obtain  $\left||\bu_h^n|-1 \right| \leq \delta_{\tau, h}$  and $ \left| |2 \bu_h^n - \bu_h^{n-1}| - 1\right| \leq 3 \delta_{\tau, h} + c \tau$. Hence, for the extrapolations defined in \eqref{extrapol} and with $\varepsilon_1, \varepsilon_2$ sufficiently small we get
\begin{equation} \label{h2}
  \| \hat\bu_h^{j+1, (k)} - \bu(t_{j})\|_{L^\infty} \leq c(\delta_{\tau,h} + \tau), \quad k=1,2.
\end{equation} 
Since $w_h=0$ on the boundary of the triangulation we have $\bv_h=I_h(\bu(t_j)w_h) \in \bV_{h,0}^p$. Using \eqref{h1} and \eqref{h2} we obtain
\begin{align*}
 & \sup_{\mathbf{v}_{h} \in \mathbf{{V}}_{h,0}^p} \mathbf{b}(\hat{\mathbf{u}}_{h}^{j+1,(k)}; \mathbf{v}_{h},w_{h}) \geq
  \mathbf{b}(\hat{\mathbf{u}}_{h}^{j+1,(k)}; I_h(\bu(t_j)w_h),w_{h}) \\ & 
  = \left( \hat{\mathbf{u}}_{h}^{j+1,(k)} \cdot I_h(\bu(t_j)w_h),w_{h} \right)_{L^2} 
  \\
  & = \left( \big(\hat{\mathbf{u}}_{h}^{j+1,(k)}- \bu(t_j)\big)\cdot I_h(\bu(t_j)w_h), w_h \right)_{L^2} +\left( \bu(t_j) \cdot \bu(t_j) w_h, w_h \right)_{L^2} \\
   &  \quad + \left( \bu(t_j) \cdot \big( I_h(\bu(t_j)w_h) - \bu(t_j)w_h \big), w_h \right)_{L^2} \\
   & \geq \|w_h\|_{L^2}^2 - c \|\hat{\mathbf{u}}_{h}^{j+1,(k)}- \bu(t_j)\|_{L^\infty} \|w_h\|_{L^2}^2 - \|I_h(\bu(t_j)w_h) - \bu(t_j)w_h
   \|_{L^2} \|w_h\|_{L^2} \\
   & \geq \big( 1 - c(\delta_{\tau,h} + \tau + h) \big)\|w_h\|_{L^2}^2 .
  \end{align*}
For $\varepsilon_1$ and  $\varepsilon_2$ sufficiently small we have $ 1 - c(\delta_{\tau,h} + \tau + h) \geq \tfrac12$. 
Finally note that $\|I_h(\bu(t_j)w_h)\|_{L^2} \leq c \|w_h\|_{L^2}$ holds. 
\end{appendices}

\newpage
\bibliographystyle{siam}
\bibliography{literature}

\end{document}